\documentclass[11pt]{amsart}
\usepackage{graphicx}
\usepackage{amsmath,amsthm,amssymb,enumerate}
\usepackage{euscript,mathrsfs}
\usepackage{amscd}
\usepackage{color}
\usepackage[citecolor=blue,colorlinks=true]{hyperref}
\usepackage[left=2cm,right=2cm,top=3.5cm,bottom=3.5cm]{geometry}
\usepackage{color}
\usepackage[framemethod=tikz]{mdframed}
\allowdisplaybreaks

\catcode`\@=11 \@addtoreset{equation}{section}

\catcode`\@=12

\allowdisplaybreaks

\newtheorem{Theorem}{Theorem}[section]
\newtheorem{Proposition}[Theorem]{Proposition}
\newtheorem{Lemma}[Theorem]{Lemma}
\newtheorem{Corollary}[Theorem]{Corollary}

\theoremstyle{definition}
\newtheorem{Definition}[Theorem]{Definition}

\newtheorem{Remark}[Theorem]{Remark}

\newcommand{\bTheorem}[1]{
\begin{Theorem} \label{T#1} }
\newcommand{\eT}{\end{Theorem}}

\newcommand{\bProposition}[1]{
\begin{Proposition} \label{P#1}}
\newcommand{\eP}{\end{Proposition}}

\newcommand{\bLemma}[1]{
\begin{Lemma} \label{L#1} }
\newcommand{\eL}{\end{Lemma}}

\newcommand{\bCorollary}[1]{
\begin{Corollary} \label{C#1} }
\newcommand{\eC}{\end{Corollary}}

\newcommand{\ep}{\varepsilon}

\DeclareMathOperator{\Div}{div}

\makeatletter
\@tfor\next:=abcdefghijklmnopqrstuvwxyzABCDEFGHIJKLMNOPQRSTUVWXYZ\do{%
\def\command@factory#1{%
\expandafter\def\csname b#1\endcsname{\mathbf{#1}}
\expandafter\def\csname cl#1\endcsname{\mathcal{#1}}
}
\expandafter\command@factory\next
}

\newcommand{\bRemark}[1]{
\begin{Remark} \label{R#1} }
\newcommand{\eR}{\end{Remark}}

\newcommand{\bDefinition}[1]{
\begin{Definition} \label{D#1} }
\newcommand{\eD}{\end{Definition}}

\newcommand{\Q}{\mathbb{T}^N}

\newcommand{\Del}{\Delta_x}

\newcommand{\Ds}{\mathbb{D}_x}
\newcommand{\p}{\mathbb{P}}

\newcommand{\mn}{\mathbb{N}}

\newcommand{\mt}{\mathbb{T}^N}

\newcommand{\bfu}{\mathbf{u}}

\newcommand{\bfq}{\mathbf{q}}

\newcommand{\bfr}{\mathbf{r}}

\newcommand{\bfA}{\mathbf{A}}
\newcommand{\bfQ}{\mathbf{Q}}
\newcommand{\bfW}{\mathbf{W}}
\newcommand{\bfV}{\mathbf{V}}

\newcommand{\bfZ}{\mathbf{Z}}
\newcommand{\bfX}{\mathbf{X}}

\newcommand{\bfphi}{\boldsymbol{\varphi}}

\newcommand{\bfpsi}{\boldsymbol{\psi}}

\newcommand{\ds}{\,\mathrm{d}\sigma}
\newcommand{\dd}{\mathrm{d}}
\newcommand{\dxs}{\,\mathrm{d}x\,\mathrm{d}\sigma}
\newcommand{\dxt}{\,\mathrm{d}x\,\mathrm{d}t}

\newcommand{\bFormula}[1]{
\begin{equation} \label{#1}}
\newcommand{\eF}{\end{equation}}

\newcommand{\vr}{\varrho}

\newcommand{\vu}{\bfu}

\newcommand{\vc}[1]{{\mathbf{#1}}}

\newcommand{\Grad}{\nabla_x}

\newcommand{\dx}{\,{\rm d} {x}}

\newcommand{\dt}{\,{\rm d} t }

\newcommand{\D}{{\rm d}}

\newcommand{\R}{\mathbb{R}}

\newcommand{\mathd}{\mathrm{d}}
\newcommand{\tmop}{\text}
\def\softd{{\leavevmode\setbox1=\hbox{d}%
          \hbox to 1.05\wd1{d\kern-0.4ex{\char039}\hss}}}
\definecolor{Cgrey}{rgb}{0.85,0.85,0.85}
\definecolor{Cblue}{rgb}{0.50,0.85,0.85}
\definecolor{Cred}{rgb}{1,0,0}
\definecolor{fancy}{rgb}{0.10,0.85,0.10}

\newcommand\Cbox[2]{%
    \newbox\contentbox%
    \newbox\bkgdbox%
    \setbox\contentbox\hbox to \hsize{%
        \vtop{
            \kern\columnsep
            \hbox to \hsize{%
                \kern\columnsep%
                \advance\hsize by -2\columnsep%
                \setlength{\textwidth}{\hsize}%
                \vbox{
                    \parskip=\baselineskip
                    \parindent=0bp
                    #2
                }%
                \kern\columnsep%
            }%
            \kern\columnsep%
        }%
    }%
    \setbox\bkgdbox\vbox{
        \color{#1}
        \hrule width  \wd\contentbox %
               height \ht\contentbox %
               depth  \dp\contentbox
        \color{black}
    }%
    \wd\bkgdbox=0bp%
    \vbox{\hbox to \hsize{\box\bkgdbox\box\contentbox}}%
    \vskip\baselineskip%
}

\mdfdefinestyle{MyFrame}{%
	linecolor=black,
	outerlinewidth=1pt,
	roundcorner=5pt,
	innertopmargin=\baselineskip,
	innerbottommargin=\baselineskip,
	innerrightmargin=10pt,
	innerleftmargin=10pt,
	backgroundcolor=white!20!white}

\date{}

\begin{document}


\title{Compressible Navier--Stokes system with transport noise}

\author{Dominic Breit}
\address[D. Breit]{
Department of Mathematics, Heriot-Watt University, Riccarton Edinburgh EH14 4AS, UK}
\email{d.breit@hw.ac.uk}

\author{Eduard Feireisl}
\address[E. Feireisl]{Institute of Mathematics AS CR, \v{Z}itn\'a 25, 115 67 Praha 1, Czech Republic
\and Institute of Mathematics, TU Berlin, Strasse des 17.Juni, Berlin, Germany }
\email{feireisl@math.cas.cz}
\thanks{The research of E.F. leading to these results has received funding from
the Czech Sciences Foundation (GA\v CR), Grant Agreement
21--02411S. The Institute of Mathematics of the Academy of Sciences of
the Czech Republic is supported by RVO:67985840. }

\author{Martina Hofmanov\'a}
\address[M. Hofmanov\'a]{Fakult\"at f\"ur Mathematik, Universit\"at Bielefeld, D-33501 Bielefeld, Germany}
\email{hofmanova@math.uni-bielefeld.de}
\thanks{M.H. has received funding from the European Research Council (ERC) under the European Union's Horizon 2020 research and innovation programme (grant agreement No. 949981). The financial support by the German Science Foundation DFG via the Collaborative Research Center SFB1283 is gratefully acknowledged.}

\author{Ewelina Zatorska}
\address[E. Zatorska]{Department of Mathematics, Imperial College London,
 London SW7 2AZ, United Kingdom }
\email{e.zatorska@imperial.ac.uk}
\thanks{The work of E.Z. was supported by the EPSRC Early Career Fellowship no. EP/V000586/1.}

\begin{abstract}
We consider the barotropic Navier--Stokes system driven by a physically well-motivated transport noise in both continuity as well as momentum equation. We focus on  three different situations:
(i) the noise is smooth in time and the equations are understood as in the sense of the classical weak deterministic theory, (ii) the noise is rough in time and we interpret the equations in the framework of rough paths with unbounded rough drivers and (iii) we have a Brownian noise of Stratonovich type and study the existence of martingale solutions. The first situation serves as an approximation for (ii) and (iii), while (ii) and (iii) are motivated by recent results on the incompressible Navier--Stokes system concerning the physical modeling
as well as regularization by noise.
\end{abstract}

\subjclass[2010]{60H15, 35R60, 76N10,  35Q35}
\keywords{Compressible fluids, stochastic Navier--Stokes system, transport noise}

\date{\today}

\maketitle

\section{Introduction}

Many SPDEs studied in the context of fluid mechanics concern fluids \emph{driven} by stochastic forcing. 
Randomness is incorporated through the effect of the outer world while the fluid model remains deterministic. 
The recent striking discoveries concerning possible ill--posedness of some well established mathematical models, notably the Euler and Navier--Stokes systems, motivated a renewed interest in a proper modification of the existing models to restore regularity and well--posedness at least at a stochastic level.

Mikulevicius and Rozovskii \cite{MiRo} introduced randomness at the Lagrangian level imposing a stochastic forcing in the equation for the streamlines $\bfX = \bfX(t)$, 
\begin{equation} \label{Mo1}
\D \bfX = \vu (t, \bfX) \dt + \sigma \circ \D \bfW,
\end{equation}
where the macroscopic velocity $\vu$ is augmented by a stochastic forcing of Stratonovich type.
Similar ideas were incorporated in the general theory developed by Holm \cite{HOLM}, see also the most recent development in \cite{ABDT}, \cite{DRHO}, \cite{HOLM1}. In the same spirit, a series of different models 
has been proposed by Cruzeiro et al. \cite{ARCR}, \cite{ChCrRa}, \cite{CRLAS}, including stochastic variants of the compressible Navier--Stokes system, see Section \ref{sec:CCR} below. 

From a more mathematical perspective the right choice of transport noise can lead to regularization effects and improve well-posedeness results for deterministic problems.  A breakthrough in this direction has been achieved in
\cite{FGP}, where it has been shown that transport noise has regularizing effects on the transport equation. A first result for the incompressible Navier--Stokes equations has been recently established in the remarkable work \cite{FL}. It is proved that a particular transport noise  delays the blow-up of the vorticity with large probability. As a matter of fact, it is not stochasticity of the forcing which provides this regularization but the roughness of the perturbation. The same effect is present in the case of a suitable purely deterministic forcing
 in the context of geometric rough paths, cf. \cite{FHLN} .

The goal of the present paper is to develop a rigorous existence theory for models of compressible fluids including the  random effects motivated by \eqref{Mo1}. In general, we consider the randomness as an intrinsic property of the fluid motion without any impact on the bulk macroscopic quantities, in particular, the \emph{total} mass, momentum, and energy. Accordingly, the bulk velocity will undergo random perturbations 
only in the convective terms, while the same quantity remains unchanged in the advected quantities and diffusion transport terms. To a certain extent, the idea is similar to the ``two velocities concept'' advocated by Brenner \cite{BREN}, \cite{BREN1}.  
The two-velocity formulation and resulting drift in the balance of mass and momentum are also postulated for the second-order Aw-Rascle model of vehicular traffic \cite{AR}. In this model the velocity of free traffic $u$ differs from the actual velocity $w$ by the velocity offset $p$. 
The model then consists of two conservation equations of mass and of the augmented momentum:
\begin{equation*}
\left\{
\begin{array}{l}
	\partial_t n + \partial_x (n u) = 0,\\ 
	\partial_t (n w) + \partial_x (n u w ) = 0,
\end{array}\right.
\end{equation*}
or equivalently
\begin{equation*}
\left\{
\begin{array}{l}
	\partial_t n + \partial_x (n w) =  \partial_x (n p),\\ 
	\partial_t (n w) + \partial_x (n w w) = \partial_x (n w p),
\end{array}\right.
\end{equation*}
where the form of the drift terms  becomes apparent on the right hand side. In the simplest setting,  the difference between the velocities depends only on the concentration of the cars, i.e., there is a ``pseudo-pressure'' function $p(n)$ such that $w=u+p(n)$. For example, in \cite{BDDR}, $p(n)\approx \frac{n}{n^*-n}$, where $n^*$ denotes some safe concentration of the cars, assumed to be constant. However, by analogy to the first order models, $n^*$ could also include: the dependence on velocity of the cars (the safe distances between the drivers should increase with the speed of driving), the random dependence on the space variable (modeling obstacles on the road), or the dependence on the initial condition and time (to take into account individual initial preferences of the interacting agents). See, for example, \cite{BDBMRR, DMNZ,PT} for a relevant overview of the models and results.
\medskip

For the sake of simplicity, we focus on the barotropic fluid ignoring the temperature changes. In accordance with the above philosophy, the field equations in the unknown velocity $\bfu$ and density $\varrho$ read 

\begin{mdframed}[style=MyFrame] 
\begin{align}
	\D \vr + \Div (\vr \vu) \dt &= \Div (\vr \mathbb{Q} ) \circ \D \bfW
	\label{p1} \\ 
	\D (\vr \vu) + \Div (\vr \vu \otimes \vu) \dt + \Grad p(\vr) \dt &= \Div \mathbb{S}(\Ds \vu) \dt + \Div (\vr \vu \otimes \mathbb{Q} ) \circ \D \bfW ,
\label{p2}
\end{align}
where 
\[
\Div (\vr \mathbb{Q} ) \D \bfW = \partial_{x_j} (\vr Q_{j,k} ) \D W_k,\ 
[ \Div (\vr \vu \otimes \mathbb{Q} ) \D \bfW ]_i = \partial_{x_j} (\vr u^i Q_{j,k} ) \D W_k,  
\]
and $\bfW$ is a $K$-dimensional  Wiener process. 

\end{mdframed}
Thus, similarly to \eqref{Mo1}, the advective component of the velocity has been augmented by a random term 
\[
\mathbb{Q} \circ \D \bfW.
\]
We suppose Newton's rheological law
\begin{align}\label{eq:nr}
\mathbb S(\Ds\bfu)=2\mu\Ds\vu+\eta\Div\bfu\,\mathbb I=2\mu\big(\nabla\bfu+\nabla\bfu^T\big)+\eta\Div\bfu\,\mathbb I
\end{align}
with strictly positive viscosity coefficients $\mu$ and $\eta$
for the viscous stress tensor, and the adiabatic pressure law $p(\varrho)=a\varrho^\gamma$ with $a>0$ and $\gamma>1$.
Finally, we impose the periodic boundary conditions 
\begin{equation} \label{p5}
	\Omega = \mathbb{T}^N,\quad N=2,3.
	\end{equation}

	In contrast to the variants of the stochastic compressible Navier--Stokes system considered so far, cf.  \cite{BrHo,BFHbook,Sm}, the noise in \eqref{p1}--\eqref{p2} is energy conservative (we outline this below in Subsection \ref{EE}). In fact, the noise considered in previous papers is constantly adding energy to the system. This is physically unreasonable and even results in the non-existence of stationary solutions for the full Navier--Stokes--Fourier system of heat-conducting fluids, cf. \cite[Section 7]{BF}.
	
	In the context of transport noise we consider three different situations:
	\begin{itemize}
\item[(i)] The noise is smooth in time and the equations are understood as in the classical weak deterministic theory from \cite{Li2,feireisl1}, see Theorem \ref{thm:main} for details.
\item[(ii)] The noise is rough in time and we interpret the equations in the framework of rough paths with unbounded rough drivers, see Theorem \ref{thm:mainrough}.
\item[(iii)] We have a Brownian noise of Stratonovich type and study the existence of weak martingale solutions in the spirit of \cite{BrHo}, see Theorem \ref{thm:mainstrat}. 
\end{itemize}
The study of the system subject to smooth noise in Section \ref{s:8} is rather an auxiliary result (which we use as an approximate system for the other cases). Its analysis is based on a three-layer approximation scheme as introduced in \cite{feireisl1}. 

Eventually, we turn to the case of a rough noise in Section \ref{s:r}. Here, the system can be driven by a general geometric 2-step rough path which gives some flexibility from the modeling point of view. One can consider a Brownian motion but also a fractional Brownian motion with Hurst parameter $H>1/3$ or other Gaussian processes which can be lifted to such rough paths. The proof uses the power of rough path theory and the analysis here is entirely deterministic. Namely, after the corresponding rough path has been built, possibly using probability theory, we fix its  realization $\omega$ and deduce existence of a weak solution. This solution is  obtained via a Wong--Zakai approximation, i.e., as limit of solutions corresponding to smooth approximate noises.

Since uniqueness is an open problem, solutions obtained this way are not stochastic processes. Indeed, measurability has  been lost by taking subsequences depending on $\omega$. We overcome this issue by combining rough path theory and probabilistic arguments in Section \ref{sec:strat}. More precisely, we consider
\eqref{p1}--\eqref{p2} perturbed by a stochastic transport noise of Stratonovich type. We employ the rough path ideas from Section \ref{s:r} and establish  existence of a martingale solution based on the stochastic compactness method using Jakubowski's extension of the Skorokhod representation theorem from \cite{jakubow}.

The key difficulty in both (ii) and (iii) is to obtain strong convergence of the density based on the effective viscous flux identity mentioned above. In the smooth setting we are able to find a commutator which gains one derivative and therefore permits to get the necessary estimate and to pass to the limit. However, this does not seem to be possible in the irregular setting unless the vector fields $\mathbb{Q}$ are independent of the spatial variable.\footnote{In view of possible extensions of our results to the Dirichlet problem, where the noise should vanish at the boundary, it would be desirable to allow $x$-dependence of $\mathbb Q$.} For constant vector fields, we prove a nice cancellation of the corresponding noise terms.
As a consequence, the strong convergence of the density can be reduced to the classical arguments from the deterministic theory from \cite{Li2,feireisl1}.

	\subsection{Energy estimates}
\label{EE}
 Let us explain on a formal level why the noise in \eqref{p1}--\eqref{p2} is energy conservative. 
Let $P$ satisfying 
\[
P'(\vr) \vr - P(\vr) = p(\vr)
\]
be the associated pressure potential.
Multiplying \eqref{p1} on $P'(\vr)$ we get 
\begin{align}
	\D P(\vr) + \Div (P(\vr) \vu) \dt + p(\vr) \Div \vu \dt = P'(\vr) \Div (\vr \mathbb{Q} ) \circ\D \bfW.
\label{e1}
\end{align}
Repeating the same with $- \frac{1}{2} |\vu|^2$ yields
\begin{align}
 - \frac{1}{2} |\vu|^2\D \vr - \frac{1}{2} |\vu|^2 \Div (\vr \vu) \dt &= - \frac{1}{2} |\vu|^2 \Div (\vr \mathbb{Q} ) \circ \D \bfW
\label{e2}
\end{align}
Finally, multiplying \eqref{p2} on $\vu$ yields
\begin{align}
\begin{aligned}
	|\vu|^2 \D \vr &+ \vr \frac{1}{2} \D |\vu|^2 + \Div (\vr \vu) |\vu|^2 \dt + \vr \vu \frac{1}{2} \Grad |\vu|^2 \dt + \Div (p(\vr) \vu) \dt - p(\vr) \Div \vu \dt \\
=&\Div (\mathbb{S}(\Ds \vu) \cdot \vu ) \dt - \mathbb{S}(\Ds \vu) : \Ds \vu \dt + \vu \cdot \Div (\vr \vu \otimes \mathbb{Q} ) \circ \D \bfW.
\label{e3}
\end{aligned}
\end{align}
Summing up \eqref{e1}--\eqref{e3} we obtain the total energy balance
\begin{align}
\begin{aligned}
	\D \left[ \frac{1}{2} \vr |\vu|^2 + P(\vr) \right] &+ \Div \left(  \left[ \frac{1}{2} \vr |\vu|^2 + P(\vr)  \right] \vu \right) \dt + 
	\Div \Big( p(\vr) \vu - \mathbb{S}(\Ds \vu) \cdot \vu \Big) \dt \\
	&= - \mathbb{S} (\Ds \vu) : \Ds \vu \dt \\ &\quad + P'(\vr) \Div (\vr \mathbb{Q} ) \circ \D \bfW - \frac{1}{2} |\vu|^2 \Div (\vr \mathbb{Q}) \circ \D \bfW.
	+ \vu \cdot \Div (\vr \vu \otimes \mathbb{Q} ) \circ \D \bfW
	\label{e4}
	\end{aligned}
	\end{align}
In addition, after a straightforward manipulation,
\[ 
- \frac{1}{2} |\vu|^2 \Div (\vr \mathbb{Q}) \circ \D \bfW
+ \vu \cdot \Div (\vr \vu \otimes \mathbb{Q} ) \circ \D \bfW = \Div \left( \frac{1}{2} \vr |\vu|^2 \mathbb{Q} \right) \circ \D \bfW,
\]
and 
\[
P'(\vr) \Div (\vr \mathbb{Q} ) \circ \D \bfW = \Div (P(\vr) \mathbb{Q} ) \circ \D \bfW + p(\vr) \Div \mathbb{Q} \circ \D \bfW.
\]

Thus the energy balance \eqref{e4} reads 
\begin{align}
\begin{aligned}
	\D \left[ \frac{1}{2} \vr |\vu|^2 + P(\vr) \right] &+ \Div \left(  \left[ \frac{1}{2} \vr |\vu|^2 + P(\vr)  \right] \vu \right) \dt + 
	\Div \Big( p(\vr) \vu - \mathbb{S}(\Ds \vu) \cdot \vu \Big) \dt \\
	&= - \mathbb{S} (\Ds \vu) : \Ds \vu \dt + \Div \left( \left[ \frac{1}{2} \vr |\vu|^2 + P(\vr) \right] \mathbb{Q} \right)  \circ \D \bfW \\
	&\quad + p(\vr) \Div \mathbb{Q} \circ \D \bfW
	\label{e5}
	\end{aligned}
\end{align}
In order to control the energy, we need 
\begin{align}
	\Div \mathbb{Q} = 0 \ \Leftrightarrow \ \partial_{x_j} Q_{j,k} = 0 \ \mbox{for any}\ k=1,\dots, K.
\label{e6}	
\end{align}
Under this additional assumption we obtain finally
$$	\D \int_{\mt}\left[ \frac{1}{2} \vr |\vu|^2 + P(\vr) \right]\dx=- \int_{\mt}\mathbb{S} (\Ds \vu) : \Ds \vu \dx\dt.$$
	
	\subsection{Comparison with Chen, Cruzeiro, Ratiu}
\label{sec:CCR}
Finally, we compare the present model with the approach of Chen, Cruzeiro, and Ratiu \cite{ChCrRa}. In \cite[Theorem 5.5]{ChCrRa} they propose the following model 
\begin{align*}
\begin{aligned}
	\D \vr + \Div (\vr \vu) \dt &= 0,\\ 
	\D \vu + \vu \cdot \Grad \vu &= - \frac{1}{\vr} \left( \sqrt{2 \mu} \Grad \vu \cdot \D \bfW - \mu \Del \vu \dt - \eta \Grad \Div \vu \dt + \Grad p(\vr) \dt \right).
	\end{aligned}
	\end{align*}
Here, in addition, the term $\Grad \vu \cdot \D \bfW$ is interpreted as 
\[
\sum_{k}\partial_{x_k} \vu \,\D W_k, \ W_k = (W,\dots, W)
\]
and the stochastic integral is It\^ o's integral. 

As the equation of continuity is deterministic, the system can be written as 
\begin{align*}
	\D \vr + \Div (\vr \vu) \dt &= 0,\\ 
	\D (\vr \vu) + \Div (\vr \vu \otimes\vu )\dt + \Grad p(\vr) \dt &= 
	\mu \Del \vu \dt + \eta \Grad \Div \vu \dt - \sqrt{2 \mu} \sum_k \partial_{x_k} \vu \,\D W.
\end{align*}
In the associated energy balance, there is the It\^ o's correction term 
\[
\frac{\mu}{\vr} \sum_k \left|\partial_{x_k} \vu_{i} \right|^2,\ i = 1,2,3,
\] 
which does not seem to be  controllable due to the appearance of the density in the denominator. So, it is unclear how to prove 
energy estimates.

\section{Smooth noise}
\label{s:8}

In this section we establish existence of a solution to \eqref{p1}--\eqref{p2} under the assumption that the noise is smooth in time. Specifically, we assume 
\begin{align}\label{noise1}
\mathbb Q&=(Q_k)_{k=1}^K,\quad Q_k\in W^{2,\infty}_{\mathrm{div}}(\mt,\R^{N}),\\
\bfW&=(W_k)_{k=1}^\infty,\quad W_k\in C^1(\overline I,\R^N).\label{noise3}
\end{align}
and set $\mathbb Q\bfW:=\sum_{k=1}^K Q_kW_k\in W^{1,\infty}(I\times\mt,\R^N)$.

A weak solution $(\bfu,\varrho)$ satisfies the system in the following sense: 
\begin{itemize}
\item\label{D1} The momentum equation \begin{align}\label{eq:mom}
\begin{aligned}
&-\int_I\int_{\mt} \Big(\varrho \bfu\cdot \partial_t \bfphi +\varrho \bfu\otimes \bfu:\nabla_x \bfphi\Big)\dxt
\\
&+\int_I\int_{\mt}\mathbb S(\nabla_x \bfu):\nabla_x \bfphi \dxt-\int_I\int_{\mt}
p(\varrho)\,\Div \bfphi \dxt
\\&=\int_{\mt}\bfq_0 \cdot \bfphi(0)\dx-\int_I\int_{\mt}\varrho \bfu\otimes\mathbb Q:\nabla_x \bfphi\,\partial_t \bfW\dxt
\end{aligned}
\end{align} 
holds for all $\bfphi\in C^\infty(\overline I\times\mt)$ with $\bfphi(T)=0$.
\item\label{D2}  The continuity equation
\begin{align}\label{eq:con}
\begin{aligned}
-\int_I\int_{\mt}\Big(\varrho\partial_t\psi
+\varrho \bfu\cdot\nabla\psi\Big)\dxt&=\int_{\mt}\varrho_0 \psi(0)\dxt\\&+\int_I\int_{\mt}\varrho \mathbb Q\cdot\nabla_x \psi\,\partial_t \bfW\dxt
\end{aligned}
\end{align}
for all $\psi\in C^\infty(\overline{I}\times\mt)$ with $\psi(T)=0$. In addition, 
the renormalized version of \eqref{eq:con} is satisfied, 
\begin{align}\label{eq:ren}
	\begin{aligned}
		-\int_I\int_{\mt}\Big(\theta(\varrho)\partial_t\psi
		+\theta(\varrho) \bfu\cdot\nabla\psi\Big)\dxt&=\int_{\mt}\theta(\varrho_0) \psi(0)\dxt\\
		&-\int_I\int_{\mt}\big(\theta(\varrho)-\theta'(\varrho)\varrho\big)\,\Div\bfu\psi\dxt
		\\&+\int_I\int_{\mt}\theta(\varrho) \mathbb Q\cdot\nabla \psi\,\partial_t \bfW\dxt
	\end{aligned}
\end{align} 
for all $\psi\in C^\infty(\overline I \times \mt)$ with $\psi(T)=0$ and all $\theta\in C^1([0,\infty))$ with $\theta'(z) \in C_c[0, \infty)$.
 
\item \label{D3} The energy inequality
\begin{align} \label{eq:ene}
\begin{aligned}
- \int_I &\partial_t \psi \,
\mathscr E \dt+\int_I\psi\int_{\mt}\mathbb S(\nabla_x \bfu):\nabla_x \bfu\dxt\leq
\psi(0) \mathscr \int_{\mt} \left( \frac{1}{2} \frac{|\vc{q}_0|^2 }{\vr_0} + P (\vr_0) \right) \dx
\end{aligned}
\end{align}
holds for any $\psi \in C^\infty_c([0, T))$. 
Here, we abbreviated
$$\mathscr E(t)= \int_{\Omega(t)}\Big(\frac{1}{2} \varrho(t) | {\bfu}(t) |^2 + P(\varrho(t))\Big)\dx, $$
where the pressure potential is given by $P(\varrho)=\frac{a}{\gamma-1}\varrho^\gamma$.
\end{itemize}
\begin{Theorem}\label{thm:main}
Let 
\[
\gamma > \frac{N}{2}.
\]
Assume that we have
\begin{align*}
\varrho_0\in L^{\gamma}(\mt),\ \vr_0 \geq 0,\ 	
\frac{|\bfq_0|^2}{\varrho_0}&\in L^1(\mt).
\end{align*}
Furthermore, suppose that $\mathbb Q$ and $\bfW$ satisfy \eqref{noise1}--\eqref{noise3}.

Then there exists
a solution $$(\bfu,\varrho)\in L^2(I;W^{1,2}(\mt))\times C_w(\overline I;L^\gamma(\mt))$$ to \eqref{eq:mom}--\eqref{eq:ene}. 

In addition, the associated pressure admits the estimate 
\begin{align}\label{eq:higherpressure}
	\int_{I\times \mt}p(\varrho)\varrho^{\Theta}\,\dd x\,\dd t\leq c,\ \Theta < \frac{2}{N} \gamma - 1,
\end{align}
where $c$ depends on the norms of $\mathbb Q$ and $\bfW$ specified in \eqref{noise1}--\eqref{noise3}. If $\mathbb Q$ is independent of $x$ estimate \eqref{eq:higherpressure} is independent of $\mathbb Q$ and $\bfW$.

\end{Theorem}

\subsection{The approximate solutions}

As the ``noise'' is smooth, the weak solutions can be constructed by means of the approximate scheme 
introduced in \cite[Chapter 7]{EF70}. First, fix 
a family of parameters
$\varepsilon>0,\,\delta>0$ and $\beta> \max\{4,\gamma\}$. Next, consider a suitable orthogonal system formed by a family of smooth functions $(\bfpsi_n)$. We choose
 $(\bfpsi_n)$ such that it is an orthonormal system with respect to the $L^2(\mathbb T^N)$ inner product which is orthogonal with respect to the the $W^{l,2}(\mathbb T^N)$ inner product.
Note that in the present setting, the basis  $(\bfpsi_n)$ can be formed by trigonometric polynomials.
Now, let us define the finite dimensional spaces
$$X_m=\mathrm{span}\{\bfpsi_1,\dots,\bfpsi_m\},\quad m\in\mn,$$
and let $P_m:L^2(\mt)\rightarrow X_m$ be the projection onto $X_m$.
We aim to find a solution $(\bfu,\varrho)$ to the following system.
\begin{itemize}
\item\label{E1} The momentum equation holds in the sense that\begin{align}\label{eq:mom:m}
\begin{aligned}
&\int_{\mt}\varrho \bfu\cdot \bfphi\dx-\int_0^t\int_{\mt}\varrho \bfu\otimes \bfu:\nabla_x \bfphi\dxs
\\
&+\int_0^t\int_{\mt}\mathbb S(\nabla_x \bfu):\nabla_x \bfphi \dxt-\int_0^t\int_{\mt}
\Big( p(\varrho) + \delta \vr^\beta \Big)\,\Div \bfphi\dxt-\varepsilon\int_0^t\int_{\mt}\varrho\bfu\cdot\Delta_x\bfphi\dxs
\\&=\int_{\mt}\bfq_0 \cdot \bfphi\dx-\int_0^t\int_{\mt}\varrho \bfu\otimes\mathbb Q:\nabla_x \bfphi\,\partial_t \bfW\dxs
\end{aligned}
\end{align} 
for all $\bfphi\in X_m$.
\item\label{E2}  The continuity equation holds in the sense that
\begin{align}\label{eq:con:m} 
\begin{aligned}
\partial_t\varrho+\Div(\varrho \bfu)=\varepsilon\Delta \varrho+\Div(\varrho\,\mathbb Q)\partial_t \bfW
\end{aligned}
\end{align}
in $I\times\mt$, and $\varrho(0)=\varrho_{0, \delta}$, where 
$\vr_{0,\delta}$ is a smooth approximation of the initial density $\vr_0$. 
\end{itemize}

Observe that the approximate system is almost the same as in \cite[Chapter 7]{EF70}, with the velocity 
in the convection term augmented by a smooth solenoidal component 
\[
\vc{w} = \mathbb{Q} \cdot \partial_t \vc{W}. 
\]
Thus the proof of convergence, consisting in three successive limits, 
\[
m \to \infty,\ \ep \to 0, \ \delta \to 0,
\]
remains essentially the same provided we clarify the following issues:
\begin{itemize}
	
	\item Total mass and energy estimates for the system with the extra drift terms. 
	\item The pressure estimates claimed in \eqref{eq:higherpressure}. 
	\item The so-called Lions' identity for the effective viscous flux.
	
	\end{itemize}

\subsection{Total mass and energy estimates} 

Obviously, the total mass 
\[
\int_{\mt} \vr \dx = M_0 
\]
remains constant, meaning determined by the initial data at any level of approximation. Moreover, as observed 
in Section \ref{EE}, the total energy is conserved even in the case of a ``non--smooth'' noise, specifically,
 \[
- \int_I \partial_t \psi \,
	\mathscr E \dt+\int_I\psi\int_{\mt}\mathbb S(\nabla_x \bfu):\nabla_x \bfu\dxt \nonumber \leq
	\psi(0) \mathscr \int_{\mt} \left( \frac{1}{2} \frac{|\vc{q}_0|^2 }{\vr_0} + P (\vr_0) \right) \dx
	\]
holds for any $\psi \in C^\infty_c([0, T))$. Here, the inequality is pertinent to weak solutions. As the energy balance is obtained, at any level of approximation, via the scalar product of the momentum equation with the velocity $\vu$, all relevant estimates remain valid in the new setting. 

%
%
%
%
%
%
%
%
Similarly to \cite[Chapter 7]{EF70} we obtain the following result.

\begin{Proposition}\label{thm:m}
Assume that we have for some $\alpha\in(0,1)$
\begin{align*}
\frac{|\bfq_0|^2}{\varrho_0}&\in L^1(\mt),\ \varrho_0\in C^{2,\alpha}(\mt).
\end{align*}
Furthermore, suppose that $\mathbb Q$ and $\bfW$ satisfy \eqref{noise1}--\eqref{noise3} and that $\varrho_0$ is strictly positive. Then there is a solution $$(\bfu,\varrho)\in C(\overline I;X_m)\times C^{\alpha}(\overline I;C^{2+\alpha}(\mt))$$ to \eqref{eq:mom:m}--\eqref{eq:con:m}. 
\end{Proposition}

\subsection{The viscous approximation}
\label{subsec:eps}
We wish to establish a solution to the following system with artificial viscosity and artificial pressure.
\begin{itemize}
\item The momentum equation holds in the sense that\begin{align}\label{eq:mom:eps}
\begin{aligned}
&-\int_{I}\int_{\mt} \Big(\varrho \bfu\cdot \partial_t \bfphi +\varrho \bfu\otimes \bfu:\nabla_x \bfphi\Big)\dxt
\\
&+\int_I\int_{\mt}\mathbb S(\nabla_x \bfu):\nabla_x \bfphi \dxt-\int_I\int_{\mt}
p_\delta(\varrho)\,\Div \bfphi\dxt-\varepsilon\int_I\int_{\mt}\varrho\bfu\cdot\Delta_x\bfphi\dxt
\\&=\int_{\mt}\bfq_0 \cdot \bfphi(0)\dx-\int_I\int_{\mt}\varrho \bfu\otimes\mathbb Q:\nabla_x \bfphi\,\partial_t \bfW\dxt
\end{aligned}
\end{align} 
for all $\bfphi\in C^\infty(\overline I\times\mt)$ with $\bfphi(T)=0$ and we have $\varrho\bfu(0)=\bfq_0$.
\item The continuity equation holds in the sense that
\begin{align}\label{eq:con:eps}
\begin{aligned}
\int_I\int_{\mt}\Big(\varrho\partial_t\psi
+\varrho \bfu\cdot\nabla\psi\Big)\dxt&=\int_{\mt}\varrho_0 \psi(0)\dx+\varepsilon\int_I\int_{\mt}\nabla\varrho\cdot\nabla\psi\dxt\\&+\int_I\int_{\mt}\varrho \mathbb Q\cdot\nabla_x \psi\,\partial_t \bfW\dxt
\end{aligned}
\end{align}
for all $\psi\in C^\infty(\overline{I}\times\mt)$ with $\psi(T)=0$ and we have $\varrho(0)=\varrho_0$. 
\item  The energy inequality is satisfied in the sense that
\begin{align} \label{eq:ene:eps}
\begin{aligned}
- \int_I \partial_t \psi \,
\mathscr E_\delta \dt&+\int_I\psi\int_{\mt}\mathbb S(\nabla_x \bfu):\nabla_x \bfu\dxs\\&+\varepsilon\int_I\psi\int_{\mt}P_\delta''(\varrho)|\nabla\varrho|^2\dxs \leq
\psi(0) \mathscr E_\delta(0)
\end{aligned}
\end{align}
holds for any $\psi \in C^\infty_c([0, T))$.
Here, we abbreviated
$$\mathscr E_\delta(t)= \int_{\mt}\Big(\frac{1}{2} \varrho(t) |\bfu(t) |^2 + P_\delta(\varrho(t))\Big)\dx$$
and the pressure potential is given by $P_\delta(\varrho)=\frac{a}{\gamma-1}\varrho^\gamma+\frac{\delta}{\beta-1}\varrho^\beta$.
\end{itemize}
\begin{Proposition}\label{thm:eps}
Assume that we have for some $\alpha\in(0,1)$
\begin{align*}
\frac{|\bfq_0|^2}{\varrho_0}&\in L^1(\mt),\ \varrho_0\in C^{2,\alpha}(\mt).
\end{align*}
Furthermore suppose thatt $\mathbb Q$ and $\bfW$ satisfy \eqref{noise1}--\eqref{noise3} and that $\varrho_0$ is strictly positive.
 There is a solution $$(\bfu,\varrho)\in L^2(I;W^{1,2}(\mt))\times C_w(\overline I;L^\beta(\mt))\cap L^2(W^{1,2}(\mt))$$ to \eqref{eq:mom:eps}--\eqref{eq:ene:eps}. 
\end{Proposition}

\begin{proof}
For a given $m\in\mathbb N$ we obtain a solution $(\bfu_m,\varrho_m)$ to \eqref{eq:mom:m}--\eqref{eq:con:m} by Theorem \ref{thm:m}. Testing
\eqref{eq:mom:m} by $\bfu_m$ and \eqref{eq:con:m} by $\frac{1}{2}|\bfu_m|^2$
we have
\begin{align*}
\frac{1}{2}\int_{\mt}\varrho_m|\bfu_m|^2\dx&+\int_0^t\int_{\mt}\mathbb S(\nabla_x\bfu_m):\nabla_x\bfu_m\dxs\\
&=\frac{1}{2}\int_{\mt}\frac{|\bfq_0^m|^2}{\varrho_0^m}\dx+\int_0^tp_\delta(\varrho_m)\,\Div\bfu_m\dxs\\
&+\int_I\int_{\mt}\varrho_m \mathbb Q\cdot\nabla_x \frac{1}{2}|\bfu_m|^2\,\partial_t \bfW\dxt-\int_I\int_{\mt}\varrho_m \bfu_m\otimes\mathbb Q:\nabla_x \bfu_m\,\partial_t \bfW\dxt\\
&=\frac{1}{2}\int_{\mt}\frac{|\bfq_0^m|^2}{\varrho_0^m}\dx+\int_0^tp_\delta(\varrho_m)\,\Div\bfu_m\dxs
\end{align*}
for almost all $0 \leq t \leq T$.
Multiplying  \eqref{eq:con:m} by $P'_\delta(\vr_m)$ we get 
\begin{align*}
	\partial_t P_\delta(\vr_m) + \Div (P_\delta(\vr_m) \vu_m)  + p_\delta(\vr_m) \Div \vu_m  = P'(\vr_m) \Div (\vr_m \mathbb{Q} ) \partial_t \bfW.
\end{align*}
For the last term we have
\begin{align*}
\int_{\mt}P'(\vr_m) \Div (\vr_m \mathbb{Q} ) \partial_t \bfW\dx=\int_{\mt}\nabla_x P(\vr_m) \cdot \mathbb{Q}  \partial_t \bfW\dx=-\int_{\mt}P(\vr_m) \,\Div \mathbb{Q}  \partial_t \bfW\dx=0
\end{align*}
such that we conclude
\begin{align}\label{NTDBkappa}
\begin{aligned}
\frac{1}{2}\int_{\mt}\varrho_m|\bfu_m|^2\dx&+\int_{\mt}P_\delta(\varrho_m)\dx+\int_0^t\int_{\mt}\mathbb S(\nabla_x\bfu_m):\nabla_x\bfu_m\dxs\\
&=\frac{1}{2}\int_{\mt}\frac{|\bfq_0^m|^2}{\varrho_0^m}\dx+\int_{\mt}P_\delta(\varrho_0^m)\dx\leq\,c.
\end{aligned}
\end{align}
We deduce the bounds
\begin{align} 
\label{est:k2}
\sup_{t \in I} \| \varrho_m \|^\beta_{L^\beta(\mt)}  +
 \sup_{t \in I} \| \varrho_m\bfu_m \|_{L^{\frac{2 \beta}{\beta + 1}}(\mt)}^{\frac{2 \beta}{\beta + 1}}  \leq c,\\
 \label{est:k3}
 \| \nabla_x \bfu_m \|^2_{L^2(I\times \mt) }+ \| \nabla_x \varrho_m \|^2_{L^2(I\times\mt) } + \| \nabla_x (\varrho_m)^{\beta/2} \|^2_{L^2(I\times\mt) } \leq c.
\end{align}

Passing to a subsequence we obtain
\begin{align}
\bfu_m&\rightharpoonup \bfu\quad\text{in}\quad L^2(I;W^{1,2}(\mt)),\label{kappaeq:convu1}\\
\varrho_m&\rightharpoonup^{\ast}\varrho\quad\text{in}\quad L^\infty(I;L^\beta(\mt)),\label{kappaeq:convrho1}\\
\varrho_m&\rightharpoonup\varrho\quad\text{in}\quad L^2(I;W^{1,2}(\mt))\label{kappaeq:convrho2}.
\end{align}
Since \eqref{eq:con:m} and \eqref{eq:mom:m} yield compactness of $\varrho_m$ and $\varrho_m\bfu_m$, is is straightforward to pass to the limit in \eqref{eq:con:m} and \eqref{eq:mom:m}. We obtain \eqref{eq:con:eps} and \eqref{eq:mom:eps}.
 Similarly, we can multiply \eqref{NTDBkappa} by a smooth temporal test-function
and use lower-semi continuity to obtain \eqref{eq:ene:eps}.
 \end{proof}

\subsection{The vanishing viscosity limit}
\label{subsec:del}
We wish to establish the existence of a weak solution $(\bfu,\varrho)$ to the system with artificial pressure in the following sense: 
\begin{itemize}
\item\label{D1} The momentum equation holds in the sense that\begin{align}\label{eq:mom:delta}
\begin{aligned}
&-\int_I\int_{\mt} \Big(\varrho \bfu\cdot \partial_t \bfphi +\varrho \bfu\otimes \bfu:\nabla_x \bfphi\Big)\dxt
\\
&+\int_I\int_{\mt}\mathbb S(\nabla_x \bfu):\nabla_x \bfphi \dxt-\int_I\int_{\mt}
p_\delta(\varrho)\,\Div \bfphi \dxt
\\&=\int_{\mt}\bfq_0 \cdot \bfphi(0)\dx-\int_I\int_{\mt}\varrho \bfu\otimes\mathbb Q:\nabla_x \bfphi\,\partial_t \bfW\dxt
\end{aligned}
\end{align} 
for all $\bfphi\in C^\infty(\overline I\times\mt)$ with $\bfphi(T)=0$ and we have $\varrho\bfu(0)=\bfq_0$;
\item\label{D2}  The continuity equation holds in the sense that
\begin{align}\label{eq:con:delta}
\begin{aligned}
-\int_I\int_{\mt}\Big(\varrho\partial_t\psi
+\varrho \bfu\cdot\nabla\psi\Big)\dxt&=\int_{\mt}\varrho_0 \psi(0)\dxt\\&+\int_I\int_{\mt}\varrho \mathbb Q\cdot\nabla_x \psi\,\partial_t \bfW\dxt
\end{aligned}
\end{align}
for all $\psi\in C^\infty(\overline{I}\times\Omega)$ with $\psi(T)=0$ and we have $\varrho(0)=\varrho_0$. 
\item \label{D3} The energy inequality is satisfied in the sense that
\begin{align} \label{eq:ene:delta}
\begin{aligned}
- \int_I &\partial_t \psi \,
\mathscr E_\delta \dt+\int_I\psi\int_{\mt}\mathbb S(\nabla_x \bfu):\nabla_x \bfu\dxt\leq
\psi(0) \mathscr E_\delta(0)
\end{aligned}
\end{align}
holds for any $\psi \in C^\infty_c([0, T))$.
Here, we abbreviated
$$\mathscr E_\delta(t)= \int_{\Omega(t)}\Big(\frac{1}{2} \varrho(t) | {v}(t) |^2 + P_\delta(\varrho(t))\Big)\dx$$
and the pressure potential is given by $P_\delta(\varrho)=\frac{a}{\gamma-1}\varrho^\gamma+\frac{\delta}{\beta-1}\varrho^\beta$.
\end{itemize}

\begin{Proposition}\label{thm:delta}
Assume that we have for some $\alpha\in(0,1)$
\begin{align*}
\frac{|\bfq_0|^2}{\varrho_0}&\in L^1(\mt),\ \varrho_0\in C^{2,\alpha}(\mt).
\end{align*}
Furthermore suppose that $\mathbb Q$ and $\bfW$ satisfy \eqref{noise1}--\eqref{noise3} and that $\varrho_0$ is strictly positive.
 There is a solution $$(\bfu,\varrho)\in L^2(I;W^{1,2}(\mt))\times C_w(\overline I;L^\beta(\mt))$$ to \eqref{eq:mom:delta}--\eqref{eq:ene:delta}.
\end{Proposition}

\begin{Lemma}
\label{cor:ap1}
Under the assumptions of Proposition \ref{thm:delta} the continuity equation holds in the renormalized sense, that is we have
\begin{align}\label{eq:ren:delta}
\begin{aligned}
-\int_I\int_{\mt}\Big(\theta(\varrho)\partial_t\psi
+\theta(\varrho) \bfu\cdot\nabla\psi\Big)\dxt&=\int_{\mt}\theta(\varrho_0) \psi(0)\dxt\\
&-\int_I\int_{\mt}\big(\theta(\varrho)-\theta'(\varrho)\varrho\big)\,\Div\bfu\psi\dxt
\\&+\int_I\int_{\mt}\theta(\varrho) \mathbb Q\cdot\nabla_x \psi\,\partial_t \bfW\dxt
\end{aligned}
\end{align} 
 for all $\psi\in C^\infty(\overline I \times \mt)$ with $\psi(T)=0$ and all $\theta\in C^1([0,\infty))$ with $\theta'(z)=0$ for all $z\geq M_\theta$.
\end{Lemma}

For a given $\varepsilon$ we obtain a solution $(\bfu_\varepsilon,\varrho_\varepsilon)$ to \eqref{eq:mom:eps}--\eqref{eq:ene:eps} by Theorem \ref{thm:eps}.
In particular, we have
\begin{align} 
\begin{aligned}
\frac{1}{2}\int_{\mt}\varrho_\varepsilon|\bfu_\varepsilon|^2\dx&\int_{\mt}\Big(\frac{a}{\gamma-1}\vr_\varepsilon^\gamma+\frac{\delta}{\beta-1}\vr_\varepsilon^\beta\Big)\dx+\int_0^t\int_{\mt}\mathbb S(\nabla_x\bfu_\varepsilon):\nabla_x\bfu_\varepsilon\dxs
\\&\leq\frac{1}{2}\int_{\mt}\frac{|\bfq_0|^2}{\varrho_0}\dx+\int_{\mt}\Big(\frac{a}{\gamma-1}\vr_0^\gamma+\frac{\delta}{\beta-1}\vr_0^\beta\Big)\dx
\end{aligned}
\label{NTDB}
\end{align}
for any $0 \leq t \leq T$.
We deduce the bounds
\begin{align} 
\label{Nbv2}
\sup_{t \in I} \| \varrho_\varepsilon \|^\beta_{L^\beta(\mt)}  +
 \sup_{t \in I} \| \varrho_\varepsilon \bfu_\varepsilon \|_{L^{\frac{2 \beta}{\beta + 1}}(\mt)}^{\frac{2 \beta}{\beta + 1}}  \leq c,\\
 \label{Nbv4}
 \| \nabla_x \bfu_\varepsilon \|^2_{L^2(I\times \mt) }+ \varepsilon\| \nabla_x (\varrho_\varepsilon)^{\beta/2} \|^2_{L^2(I\times\mt) } \leq c.
\end{align}
Finally, we deduce from the equation of continuity (\ref{eq:con:eps}) (one easily proves that $\varrho_\varepsilon$ is and admissible test function by parabolic maximum regularity theory) that
\begin{equation} \label{Nbv6}
\int_{\mt} \varrho_\varepsilon(t, \cdot) \dx = \int_{\mt} \varrho_0 \dx,\
 \| \sqrt{\varepsilon} \nabla_x \varrho_\varepsilon \|_{L^2(I\times\mt)}  \leq c.
\end{equation}
Note that all estimates are independent of $\varepsilon$.
Hence, we may take a subsequence such that
\begin{align}
\bfu_\varepsilon&\rightharpoonup \bfu\quad\text{in}\quad L^2(I;W^{1,2}(\mt)),\label{eq:convu1}\\
\varrho_\varepsilon&\rightharpoonup^{\ast}\varrho\quad\text{in}\quad L^\infty(I;L^\beta(\mt)),\label{eq:convrho1}\\
\varepsilon\nabla\varrho_\varepsilon&\rightarrow 0\quad\text{in}\quad L^2(I\times\mt).\label{eq:van1}
\end{align}
We observe that the a priori estimates \eqref{Nbv2} imply uniform bounds of
$\varrho_\varepsilon \bfu_\varepsilon$ in $ L^\infty(I,L^\frac{2\beta}{\beta+1}(\mt))$. Therefore, we may obtain using \eqref{eq:con:eps} in conjunction with \eqref{eq:convu1}
\begin{align}
\varrho_\varepsilon \bfu_\varepsilon&\rightharpoonup {\varrho}  \bfu\qquad\text{in}\qquad L^q(I, L^a(\mt)),\label{conv:rhov2}
\end{align}
where $a\in (1,\frac{2\beta}{\beta+1})$ and $q\in (1,2)$.
Similarly, we can use \eqref{eq:mom:eps} to conclude
\begin{align}
{\varrho}_\varepsilon  \bfu_\varepsilon\otimes \bfu_\varepsilon&\rightharpoonup  {\varrho}  \bfu\otimes  \bfu\quad\text{in}\quad L^1(I\times\mt).\label{conv:rhovv2}
\end{align} 
At this stage of the proof the pressure is only bounded in $L^1$, so we have to exclude its concentrations. We are going to prove that
\begin{equation}\label{eq:gamma+1}
\int_{I\times \mt}p_\delta(\varrho_\varepsilon)\varrho_{\varepsilon}\,\dd x\,\dd t\leq c
\end{equation}
with a constant independent of $\varepsilon$.
In order to verify \eqref{eq:gamma+1} we test the momentum equation \eqref{eq:mom:eps} with $\nabla_x\Delta^{-1}_x\varrho_\varepsilon$.
In order to deal with the term involving the time derivative we use the continuity equation \eqref{eq:con:eps}. It holds
$$\partial_t\nabla_x\Delta^{-1}_x\varrho_\varepsilon=\nabla\Delta^{-1}_x\Div(\varrho_\varepsilon \bfu_\varepsilon+\varepsilon\nabla_x\varrho_\varepsilon )+\nabla_x\Delta^{-1}_x\Div(\varrho_\varepsilon\mathbb Q\partial_t\bfW).$$ 
such that we obtain
\begin{align}\label{eq:testtheta}
\begin{aligned}
J_0&:=\int_I\int_ {\mt} p_\delta(\vr_\varepsilon)\varrho_\varepsilon\dxs
\\
&=\mu\int_I\int_ {\mt}\mathbb S( \nabla_x \bfu_\varepsilon):\nabla^2_x\Delta^{-1}_{x} \varrho_\varepsilon\dxs
-\int_I\int_ {\mt}  \varrho_\varepsilon \bfu_\varepsilon\otimes \bfu_\varepsilon:\nabla_x^2\Delta^{-1}_{x}\varrho_\varepsilon\dxs
\\
&+\varepsilon\int_I\int_{\mt}\nabla(\bfu_\varepsilon\varrho_\varepsilon):\nabla^2_x\Delta^{-1}_{x}\varrho_\varepsilon\dxs+\int_I\int_{\mt}\varrho_\varepsilon\bfu_\varepsilon\nabla\Delta^{-1}_{x}\Div(\varrho_\varepsilon\bfu_\varepsilon+\varepsilon \nabla_x \varrho_\epsilon)\dxs
\\
& +\int_I\int_ {\mt} \Big[  \Del^{-1} \Grad [\vr_\varepsilon] \cdot \Div (\vr_\varepsilon \vu_\varepsilon \otimes \mathbb{Q} ) + 	\vr_\varepsilon \vu_\varepsilon \cdot \Del^{-1} \Grad  \Div (\vr_\varepsilon \mathbb{Q} ) \Big] \partial_t \bfW\dxs
\\=&:J_1+\cdots +J_{5}.
\end{aligned}
\end{align}
Based on \eqref{eq:convu1}--\eqref{eq:van1} it is well-known how to estimate the terms $J_1-J_4$.
As far as $J_5$ is concerned we rewrite
\begin{align}
\nonumber
J_5
	&= \int_I\int_ {\mt}\partial_{x_j} \Big[ \Del^{-1} \partial_{x_i} [\vr_\varepsilon] \vr_\varepsilon u^i_\varepsilon Q_{j,k} \Big] \partial_t W_k\dxs \\ 
	\nonumber&+ \int_I\int_ {\mt}\left[ u^i_\varepsilon \Big( \vr \partial_{x_j} \Del^{-1} \partial_{x_i} [ \vr_\varepsilon Q_{j,k}] - \vr_\varepsilon Q_{j,k} \partial_{x_j} \Del^{-1} \partial_{x_i} [\vr_\varepsilon] \Big) \right] \partial_t W_k\dxs\\
	&= \int_I\int_ {\mt}\left[ u^i_\varepsilon \Big( \vr_\varepsilon \partial_{x_j} \Del^{-1} \partial_{x_i} [ \vr_\varepsilon Q_{j,k}] - \vr_\varepsilon Q_{j,k} \partial_{x_j} \Del^{-1} \partial_{x_i} [\vr_\varepsilon] \Big) \right] \partial_t W_k\dxs.
	\label{P2}
\end{align}
We obtain by continuity of $\partial_{x_j} \Del^{-1} \partial_{x_i}$
\begin{align*}
|J_5|\leq\|\bfu_\varepsilon\|_{L^2(I;L^2(\mt))}\|\vr_\varepsilon\|_{L^2(I;L^4(\mt))}^2,
\end{align*}
which is uniformly bounded by \eqref{Nbv2} and \eqref{Nbv4} as long as $\beta\geq 4$. This finishes the proof of
\eqref{eq:gamma+1} and we conclude there exists a function $\overline p$ such that
\[
p_\delta(\varrho_\varepsilon)\rightharpoonup \overline p\text{ in }L^1(I;L^1(\mt)),
\]
at least for a subsequence.
Combining this with the convergences \eqref{eq:convu1}--\eqref{conv:rhovv2} we can pass to the limit in \eqref{eq:mom:eps} and \eqref{eq:con:eps} and obtain the continuity equation 
\begin{align}\label{eq:apvarrho}
\begin{aligned}
-\int_I\int_{\mt}\Big(\varrho\partial_t\psi
+\varrho \bfu\cdot\nabla\psi\Big)\dxt&=\int_{\mt}\varrho_0 \psi(0)\dxt\\&+\int_I\int_{\mt}\varrho \mathbb Q\cdot\nabla_x \psi\,\partial_t \bfW\dxt
\end{aligned}
\end{align}
for all $\psi\in C^\infty(\overline I \times \mt)$ with $\psi(T)=0$ and the momentum equation
\begin{align}\label{eq:apulim}
\begin{aligned}
&-\int_I\int_{\mt} \Big(\varrho \bfu\cdot \partial_t \bfphi +\varrho \bfu\otimes \bfu:\nabla_x \bfphi\Big)\dxt
\\
&+\int_I\int_{\mt}\mathbb S(\nabla_x \bfu):\nabla_x \bfphi \dxt-\int_I\int_{\mt}
\overline p\,\Div \bfphi \dxt
\\&=\int_{\mt}\bfq_0 \cdot \bfphi(0)\dx-\int_I\int_{\mt}\varrho \bfu\otimes\mathbb Q:\nabla_x \bfphi\,\partial_t \bfW\dxt
\end{aligned}
\end{align} 
 for all $\bfphi\in C^\infty(\overline I \times \mt)$ with $\bfphi(T)=0$.
It remains to show strong convergence of $\varrho_\varepsilon$.
The proof of strong convergence of the density is based on the effective viscous flux identity introduced in \cite{Li2} and the concept of renormalized solutions from \cite{DL}. We aim to prove that
\begin{align}\label{eq:fluxpsi}
\begin{aligned}
\int_{I\times \mt} &\big(p_\delta(\varrho_\varepsilon)-(\lambda+2\mu)\Div \bfu_\varepsilon\big)\,\varrho_\varepsilon\dxt\\&\longrightarrow\int_{I\times\mt} \big( \overline{p}-(\lambda+2\mu)\Div \bfu\big)\,\varrho\dxt
\end{aligned}
\end{align}
as $\varepsilon\rightarrow0$. This is based on rearranging the terms in \eqref{eq:testtheta} and passing to the limit in
the corresponding terms on the right-hand side (the term which one obtains when testing \eqref{eq:apulim} with $\nabla_x\Delta_x^{-1}\varrho$). This is well-known for all terms except for $J_5$. However, we can use the commutator structure in $J_5$ from \eqref{P2}. By div-curl lemma (in the version of \cite[Lemma 3.4]{feireisl1}) the convergences \eqref{eq:convu1}--\eqref{eq:convrho} are sufficient to infer
that $J_5$ converges to its expected counterpart.
This concludes the proof of \eqref{eq:fluxpsi}.\\

In order to proceed we need to derive the renormalized equation of continuity. We apply a spatial mollification with radius $\kappa\ll1$ to \eqref{eq:apvarrho} and obtain
\begin{align*}
\partial_t(\varrho)_\kappa+\Div((\varrho)_\kappa\bfu)&=\Div((\varrho)_\kappa\mathbb Q)\partial_t\bfW+\bfr_\kappa^1+\bfr_\kappa^2,\\
\bfr_\kappa^1&=\Div\big((\varrho)_\kappa\bfu-(\varrho\bfu)_\kappa\big),\\
\bfr_\kappa^2&=\Div\big((\varrho)_\kappa\mathbb Q-(\varrho\mathbb Q)_\kappa\big)\partial_t\bfW,
\end{align*}
in $I\times\mt$. Here we have
\begin{align*}
\|\bfr_\kappa^1\|_{L^q(\mt)}&\leq\,\|\bfu\|_{W^{1,2}(\mt)}\|\vr\|_{L^{\beta+1}(\mt)},\quad\frac{1}{q}=\frac{1}{2}+\frac{1}{\beta+1},\\
\|\bfr_\kappa^2\|_{L^q(\mt)}&\leq\,\|\mathbb Q\partial_t\bfW\|_{W^{1,2}(\mt)}\|\vr\|_{L^{\beta+1}(\mt)},
\end{align*}
as well as $\bfr_\kappa^1,\bfr_\kappa^2\rightarrow0$ in $L^1(\mt)$ for a.a. $t$. Hence we have $\bfr_\kappa^1,\bfr_\kappa^2\rightarrow0$ in $L^1(I\times\mt)$.
For a function $\theta\in C^1([0,\infty))$ with $\theta'(z)=0$ for all $z\geq M_\theta$ we obtain using $\Div(\mathbb Q)_\kappa=(\Div\mathbb Q)_\kappa=0$
\begin{align*}
\partial_t\theta((\varrho)_\kappa)+\Div(\theta((\varrho)_\kappa)\bfu)&=\Div(\theta((\varrho)_\kappa)\mathbb Q)\partial_t\bfW+ (\bfr_\kappa^1+\bfr_\kappa^2)\theta'((\vr)_\kappa).
\end{align*}
Multiplying by $\psi\in C^\infty(\overline I \times \mt)$ with $\psi(T)=0$ integrating in space-time and passing to the limit yields
\begin{align}\label{eq:ren:delta}
\begin{aligned}
-\int_I\int_{\mt}\Big(\theta(\varrho)\partial_t\psi
+\theta(\varrho) \bfu\cdot\nabla\psi\Big)\dxt&=\int_{\mt}\theta(\varrho_0) \psi(0)\dxt\\
&-\int_I\int_{\mt}\big(\theta(\varrho)-\theta'(\varrho)\varrho\big)\,\Div\bfu\dxt
\\&+\int_I\int_{\mt}\theta(\varrho) \mathbb Q\cdot\nabla_x \psi\,\partial_t \bfW\dxt.
\end{aligned}
\end{align} 
By the monotonicity of the mapping $\varrho\mapsto p(\varrho)$, we find that
\begin{align*}
(\lambda+2\mu)\liminf_{\varepsilon\rightarrow0}&\int_{I\times\mt} \big(\Div \bfu_\varepsilon\,\varrho_\varepsilon -\Div \bfu\,\varrho\big)\dxt\\
=&\liminf_{\varepsilon\rightarrow0}\int_{I\times\mt} \big(p(\varrho_\varepsilon)- \overline{p}\big)\big(\varrho_\varepsilon-\varrho\big)\dxt\geq 0
\end{align*}
using \eqref{eq:fluxpsi} (together with the convergences \eqref{eq:convu1} and \eqref{eq:convrho1}).
We conclude
\begin{align}\label{8.12}
\overline{\Div \bfu\,\varrho}\geq \Div \bfu\,\varrho \quad\text{a.e. in }\quad I\times\mt,
\end{align}
where 
\begin{align*}
\Div \bfu_\varepsilon\,\varrho_\varepsilon\rightharpoonup \overline{\Div \bfu\,\varrho}\quad\text{in}\quad L^1(I\times\mt),
\end{align*}
recall \eqref{eq:convu1} and \eqref{eq:convrho1}. Now, we compute both sides of
\eqref{8.12} by means of the corresponding continuity equations. Since $(\bfu_\varepsilon,\vr_\varepsilon)$ is a strong solution to
\eqref{eq:con:eps} (which can be shown by parabolic maximum regularity theory) it is also a renormalized solution and we have
\begin{align*}
\partial_t\theta(\varrho_\varepsilon)+\Div(\theta(\varrho_\varepsilon)\bfu)&=(\theta(\varrho_\varepsilon)-\theta'(\varrho_\varepsilon)\varrho_\varepsilon)+\Div(\theta(\varrho_\varepsilon)\mathbb Q)\partial_t\bfW\\
&+\Delta_x\theta(\vr_\varepsilon)-\theta''(\vr_\varepsilon)|\nabla\vr_\varepsilon|^2
\end{align*} 
for any $\theta\in C^2([0,\infty))$.
Choosing $\theta(z)=z\ln z$ and integrating in space-time we gain
\begin{align}\label{8.15}
\int_0^{t}\int_{\mt}\Div \bfu_{\varepsilon}\,\varrho_{\varepsilon}\dxs \leq\int_{\mt}\varrho_0\ln(\varrho_0)\dx
-\int_{\mt}\varrho_\varepsilon(t)\ln(\varrho_\varepsilon(t)\dx
\end{align}
for almost all $0 \leq t < T$.
Similarly, equation \eqref{eq:ren:delta} yields with the choice $\psi=\mathbb I_{[0,t]}$
\begin{align}\label{8.14}
\int_0^{t}\int_{\mt}\Div \bfu\,\varrho\dxs=\int_{\Omega_0}\varrho_0\ln(\varrho_0)\dx
-\int_{\mt}\varrho(t)\ln(\varrho(t))\dx.
\end{align}
Combining \eqref{8.12}--\eqref{8.14} shows
\begin{align*}
\limsup_{\varepsilon\rightarrow0}\int_{\mt}\varrho_\varepsilon(t)\ln(\varrho_\varepsilon(t))\dx\leq \int_{\mt}\varrho(t)\ln(\varrho(t))\dx
\end{align*}
for any $t\in I$.
This gives the claimed convergence $\varrho_{\varepsilon}\rightarrow\varrho$ in $L^1(I\times\mt)$ by convexity of $z\mapsto z\ln z$. Consequently, we have $\overline p=p(\varrho)$ and the proof of Theorem \ref{thm:delta} is complete.

\subsection{Proof of Theorem~\ref{thm:main}.}
\label{sec:6}

In this subsection we are ready to prove the main result of this section by passing to the limit $\delta\rightarrow0$ in the system \eqref{eq:mom:delta}--\eqref{eq:ene:delta} from Section \ref{subsec:del}.
Given initial data $(\bfq_0,\varrho_0)$ belonging to the function spaces stated in Theorem \ref{thm:main}
it is standard to find regularized versions $\bfq_0^{\delta}$ and $\varrho_0^{\delta}$ such that for all $\delta>0$
\begin{align*}
 \varrho_0^{\delta}\in C^{2,\alpha}(\mt),\ \varrho_0^{\delta}\ \text{strictly positive}, \int_{\mt} \varrho_0^{\delta} \dx = \int_{\mt} \varrho_0 \dx
 \end{align*}
 as well as $\bfq_0^{\delta} \to \bfq_0$ in $L^{\frac{2\gamma}{\gamma+1}}(\mt)$, $\varrho_0^{\delta} \to \varrho_0$ in $L^\gamma(\mt)$ and
 \begin{align*}
 \int_{\mt}\Big(\frac{1}{2} \frac{| {\bfq}_0^{\delta} |^2}{\varrho_0^{\delta}} &+ P_{\delta} (\varrho_0^{\delta})\Big)\dx\rightarrow  \int_{\mt}\Big(\frac{1}{2} \frac{| {\bfq}_{0} |^2}{\varrho_0} + P(\varrho_{0})\Big)\dx,
 \end{align*}
 as $\delta\rightarrow0$.
For a given $\delta$ we gain a weak solution $(\bfu_\delta,\varrho_\delta)$ to \eqref{eq:mom:delta}--\eqref{eq:ene:delta} with this data by Proposition~\ref{thm:delta}. Exactly as in Section \ref{subsec:del} we  deduce the following uniform bounds from the energy inequality:
\begin{equation} \label{wWS47}
 \sup_{t \in I} \| \varrho_\delta \|^{\gamma}_{L^{\gamma}(\mt)}  +
  \sup_{t \in I} \delta \|  \varrho_\delta \|^\beta_{L^\beta(\mt)}  
\leq c,
\end{equation}
\begin{equation} \label{wWS48}
\begin{split}
   \sup_{t \in I} \big\| \varrho_\delta |\bfu_\delta|^2 \big\|_{L^1(\mt)} +
 \sup_{t \in I} \big\| \varrho_\delta \bfu_{\delta} \big\|^\frac{2\gamma}{\gamma+1}_{L^{\frac{2\gamma}{\gamma+1}}(\mt)}   \leq c,
\end{split}
\end{equation}
\begin{equation} \label{wWS49}
 \big\| \bfu_\delta \big\|^{2}_{L^2(I;W^{1,2}(\mt))}  \leq c.
\end{equation}
Finally, we have the conservation of mass principle resulting from the continuity equation, i.e.,
\begin{equation} \label{wWS411}
\| \varrho_{\delta}(\tau, \cdot) \|_{L^1(\mt)} = \int_{\mt} \varrho(\tau, \cdot) \dx = \int_{\mt} \varrho_0 \dx  \quad \mbox{for all}\ \tau\in[0,T].
\end{equation}
Hence we may take a subsequence, such that
\begin{align}
\bfu_{\delta}&\rightharpoonup \bfu\quad\text{in}\quad L^2(I;W^{1,2}(\mt)),\label{eq:convu}\\
\varrho_{\delta}&\rightharpoonup^{\ast}\varrho\quad\text{in}\quad L^\infty(I;L^\gamma(\mt)).\label{eq:convrho}
\end{align}
Arguing as in Section \ref{subsec:del}, we find for all $q\in (1,\frac{6\gamma}{\gamma+6})$ that
\begin{align}
\varrho_{\delta}\bfu_{\delta}&\rightharpoonup {\varrho}  {\bfu}\quad\text{in}\quad L^2(I, L^q(\mt))\label{conv:rhov2delta}\\
{\varrho}_{\delta}  {\bfu}_{\delta}\otimes  {\bfu}_{\delta}&\rightarrow  {\varrho}  {\bfu}\otimes  {\bfu}\quad\text{in}\quad L^1(I;L^1(\mt)).\label{conv:rhovv2delta}
\end{align}
As before in \eqref{eq:gamma+1} we have higher integrability of the density in the sense that for $0<\Theta\leq  \frac{2}{N}\gamma-1$
\begin{equation}\label{eq:gamma+1'}
\int_{I\times \mt}p_\delta(\varrho_{\delta})\varrho_{\delta}^{\Theta}\,\dd x\,\dd t\leq c
\end{equation}
with constant independent of $\delta$. In order to prove
\eqref{eq:gamma+1'} we test the momentum equation
\eqref{eq:mom:delta} by $\Delta_x^{-1}\nabla_x\vr^\Theta$.
Noticing that
$$\partial_t\Delta^{-1}_x\nabla_x\varrho_\delta^\Theta=\nabla\Delta^{-1}_x\Div(-\varrho_\delta^\Theta \bfu_\varepsilon)+(1-\Theta)\varrho_\delta^\Theta\Div\bfu_\delta+\nabla_x\Delta^{-1}_x\Div(\varrho_\delta^\Theta\mathbb Q\partial_t\bfW)$$ 
as a consequence of the renormalized equation of continuity \eqref{eq:ren:delta}
the terms arising from the noise are (these must be controlled in addition to the known estimates)
\begin{align}
\begin{aligned}
&\int_I\int_ {\mt}\Big[  \Del^{-1} \Grad [\vr^\Theta] \cdot \Div (\vr_\delta \vu_\delta \otimes \mathbb{Q} ) + 	\vr_\delta \vu_\delta \cdot \Del^{-1} \Grad  \Div (\vr^\Theta_\delta \mathbb{Q} ) \Big] \partial_t \bfW\dxs\\
	&= \int_I\int_ {\mt}\partial_{x_j} \Big[ \Del^{-1} \partial_{x_i} [\vr_\delta^\Theta] \vr_\delta u^i_\delta Q_{j,k} \Big] \partial_t W_k\dxs \\
	&+ \int_I\int_ {\mt}\left[ u^i_\delta \Big( \vr_\delta \partial_{x_j} \Del^{-1} \partial_{x_i} [ \vr^\Theta_\delta Q_{j,k}] - \vr_\delta Q_{j,k} \partial_{x_j} \Del^{-1} \partial_{x_i} [\vr_\delta^\Theta] \Big) \right] \partial_t W_k\dxs\\
	&= \int_I\int_ {\mt}\left[ u^i_\delta \Big( \vr_\delta \partial_{x_j} \Del^{-1} \partial_{x_i} [ \vr_\delta^\Theta Q_{j,k}] - \vr_\delta Q_{j,k} \partial_{x_j} \Del^{-1} \partial_{x_i} [\vr_\delta^\Theta] \Big) \right] \partial_t W_k\dxs\\
	&\leq\|\mathbb Q\partial_t\bfW\|_{L^\infty(I\times \mt)}\|\bfu_\delta\|_{L^1(I;L^6(\mt))}\|\vr_\delta\|_{L^\infty(I;L^\gamma(\mt))}\|\vr_\delta^\Theta\|_{L^\infty(I;L^q(\mt))}
	\end{aligned}
	\label{P2delta}
\end{align}
with $q=\frac{6\gamma}{6\gamma-\gamma-6}$. Note that the integral above cancels if $\mathbb Q$ is independent of $x$ such that \eqref{eq:gamma+1'} is independent of $\mathbb Q$ and $\bfW$ in this case. In general,
it is bounded by \eqref{wWS47}--\eqref{wWS49} as a consequence of our choice of $\Theta$ and the assumptions on $\mathbb Q$ and $\bfW$. This proves \eqref{eq:gamma+1'}
 which yields the existence of a function $\overline p$ such that (for a subsequence)
\begin{align}\label{eq:limp'}
p_\delta(\varrho^{(\delta)})\rightharpoonup\overline p\quad\text{in}\quad L^{1}(I\times\mt),\\
\label{1301}
\delta(\varrho^{(\delta)})^{\beta}\rightarrow0\quad\text{in}\quad L^1(I\times\mt).
\end{align}

Using \eqref{eq:limp'} and the convergences \eqref{eq:convu}--\eqref{conv:rhovv2delta} we can pass to the limit in \eqref{eq:mom:delta} and \eqref{eq:con:delta} and obtain
\begin{align}\label{eq:apulim'}
\begin{aligned}
&-\int_I\int_{\mt} \Big(\varrho \bfu\cdot \partial_t \bfphi +\varrho \bfu\otimes \bfu:\nabla_x \bfphi\Big)\dxt
\\
&+\int_I\int_{\mt}\mathbb S(\nabla_x \bfu):\nabla_x \bfphi \dxt-\int_I\int_{\mt}
\overline p\,\Div \bfphi \dxt
\\&=\int_{\mt}\bfq_0 \cdot \bfphi(0)\dx-\int_I\int_{\mt}\varrho \bfu\otimes\mathbb Q:\nabla_x \bfphi\,\partial_t \bfW\dxt
\end{aligned}
\end{align} 
for all test-functions $\bfphi\in C^\infty(\overline I\times\mt)$ with $\bfphi(T)=0$. Moreover, the equation of continuity
\begin{align}\label{eq:apvarrholim}
-\int_I\int_{\mt}\Big(\varrho\partial_t\psi
+\varrho \bfu\cdot\nabla\psi\Big)\dxt&=\int_{\mt}\varrho_0 \psi(0)\dxt\\&+\int_I\int_{\mt}\varrho \mathbb Q\cdot\nabla_x \psi\,\partial_t \bfW\dxt
\end{align}
hold for all $\psi\in C^\infty(\overline{I}\times \mt)$ with $\psi(T)=0$.

It remains to show strong convergence of $\varrho^{(\delta)}$. 
We define the $L^\infty$-truncation
\begin{align}\label{eq:Tk'}
T_k(z):=k\,T\Big(\frac{z}{k}\Big)\quad z\in\R,\,\, k\in\mathbb N.
\end{align}
Here $T$ is a smooth concave function on $\R$ such that $T(z)=z$ for $z\leq 1$ and $T(z)=2$ for $z\geq3$. we clearly have
\begin{align}
 T_k(\varrho_{\delta})&\rightharpoonup {T}^{1,k}\quad\text{in}\quad C_w(\overline I;L^p( \mt))\quad\forall p\in[1,\infty),\label{eq:Tk1'}\\
\big(T_k'(\varrho_{\delta})\varrho_{\delta}-T_k(\varrho_{\delta})\big)\Div \bfu_{\delta}&\rightharpoonup{T}^{2,k}
\quad\text{in}\quad L^2(I\times\mt),\label{eq:Tk2'}
\end{align}
for some limit functions ${T}^{1,k}$ and ${T}^{2,k}$.
Now we have to show that 
\begin{align}\label{eq:flux'}
\begin{aligned}
\int_{I\times\mt}&\big( p_\delta(\varrho_{\delta})-(\lambda+2\mu)\Div \bfu_{\delta}\big)\,T_k(\varrho_{\delta})\dxt\\&\longrightarrow\int_{I\times\mt} \big( \overline{p}-(\lambda+2\mu)\Div \bfu\big)\,T^{1,k}\dxt.
\end{aligned}
\end{align}
In order to prove \eqref{eq:flux'} we test \eqref{eq:mom:delta} by $\Delta_x^{-1}\nabla_xT_k(\vr_\delta)$, while \eqref{eq:apulim'} is tested by
$\Delta_x^{-1}\nabla_xT^{1,k}$. The crucial point here, which makes the difference to known theory, is the prove the convergence of the terms relating to the noise: we have to show justify that
\begin{align}\label{lim:noiseflux}
\begin{aligned}
&\int_I\int_ {\mt}\Big[  \Del^{-1} \Grad [T_k(\vr_\delta)] \cdot \Div (\vr_\delta \vu_\delta \otimes \mathbb{Q} ) + 	\vr_\delta \vu_\delta \cdot \Del^{-1} \Grad  \Div (T_k(\vr_\delta) \mathbb{Q} ) \Big] \partial_t \bfW\dxs\\
&\rightarrow \int_I\int_ {\mt}\Big[  \Del^{-1} \Grad [T^{1,k}] \cdot \Div (\vr \vu \otimes \mathbb{Q} ) + 	\vr \vu \cdot \Del^{-1} \Grad  \Div (T^{1,k} \mathbb{Q} ) \Big] \partial_t \bfW\dxs
\end{aligned}
\end{align}
as $\delta\rightarrow0$. As in \eqref{P2delta} we can rewrite the term in question as
\begin{align*}
&= \int_I\int_ {\mt}\left[ u^i_\delta \Big( \vr_\delta \partial_{x_j} \Del^{-1} \partial_{x_i} [ T_k(\vr_\delta) Q_{j,k}] - \vr_\delta Q_{j,k} \partial_{x_j} \Del^{-1} \partial_{x_i} [T_k(\vr_\delta)] \Big) \right] \partial_t W_k\dxs
\end{align*}
such that the convergence in \eqref{lim:noiseflux} follows from
from \eqref{eq:convu}, \eqref{eq:convrho} and \eqref{eq:Tk1'} using the div-curl lemma (in the version of \cite[Lemma 3.4]{feireisl1}).
The next aim is to prove that $\varrho$ is a renormalized solution. Using \eqref{eq:ren:delta} with $\theta=T_k$ and passing to the limit $\delta\rightarrow0$ we arrive at
\begin{align}\label{eq:Tk''}
\partial_t T^{1,k}+\Div\big( T^{1,k}\bfu\big)+T^{2,k}= \Div\big(T^{1,k} \mathbb Q\big)\,\partial_t \bfW
\end{align}
in the sense of distributions in $I\times\mt$. Note that we extended
$\varrho$ by zero to $\R^n$. 
The next step is to show for some $q>0$
\begin{align}\label{eq.amplosc''}
\limsup_{\delta\rightarrow0}\int_{I\times\mt}|T_k(\varrho_{\delta})-T_k(\varrho)|^{q}\dxt\leq C,
\end{align}
where $C$ does not depend on $k$. The proof of \eqref{eq.amplosc''} follows exactly the arguments from the classical setting
 (see \cite{feireisl1}) using \eqref{eq:flux'} and the uniform bounds on $\bfu_\delta$. Applying a smoothing procedure (as outlined in the proof of \eqref{eq:ren:delta}) to \eqref{eq:Tk''}
 we have
 \begin{align}\label{eq:Tk''ren}
\partial_t \theta(T^{1,k})+\Div\big( \theta(T^{1,k})\bfu\big)+\theta'(T^{1,k})T^{2,k}= \Div\big(\theta(T^{1,k}) \mathbb Q\big)\,\partial_t \bfW
\end{align}
in the sense of distributions
for all $\theta\in C^1([0,\infty))$ with $\theta'(z)=0$ for all $z\geq M_\theta$. Based on \eqref{eq.amplosc''}
one can now prove that $T^{1,k}\rightarrow \varrho$
and $\theta'(T^{1,k})T^{2,k}\rightarrow0$ as $k\rightarrow\infty$.
One can now use the argument from \cite[Section 7.3]{feireisl1} to conclude the proof of strong convergence of the density.
This is not affected by the noise as it disappears after integrating \eqref{eq:Tk''ren} in space.

\section{Rough noise}
\label{s:r}

The purpose of this section is to prove an existence result for the compressible Navier--Stokes system \eqref{p1}--\eqref{p1}
driven by a rough transport noise. Here, we assume that the vector fields $\mathbb{Q}$ are independent of the spatial variable, i.e. $\mathbb{Q}=(\mathbf{Q}_{k})_{k=1}^{K}\subset R^{N}$. Before we give a rigorous definition of a solution in Section~\ref {sec:roughNS}
we introduce the set-up concerning rough paths.

\subsection{Rough paths} \label{ss:rp}

In this section 
we introduce the notion of a rough path. For an  introduction to the theory of rough paths, we refer the reader to the monographs \cite{MR2314753, FrVi10, FrHa14}. 

For a given interval $I,$ we define $\Delta_I := \{(s,t)\in I^2: s\le t\}$ and $\Delta^{(2)}_I := \{(s,\theta,t)\in I^3: s\le\theta\le t\}$.  For a given $T>0$, we let $\Delta_T := \Delta_{[0,T]}$ and $\Delta^{(2)}_T=\Delta^{(2)}_{[0,T]}$. Let $\mathcal{P}(I)$ denote the set of all partitions of an interval $I$ and let $E$ be a Banach space with norm $\| \cdot\|_E$. 
A two-index map $g: \Delta_I \rightarrow E$ is said to have finite $p$-variation for some $p>0$ on $I$ if 
$$
\|g\|_{p-\textnormal{var};I;E}:=\sup_{(t_i)\in \clP(I)}\left(\sum_{i}\|g _{t_i t_{i+1}}\|^p_E\right)^{\frac{1}{p}}<\infty.
$$
We denote by $C_2^{p-\textnormal{var}}(I; E)$ the set of all two-index maps with finite $p$-variation on $I$ equipped with the seminorm $\|\cdot\|_{p-\textnormal{var}; I;E}$. In this section,  we drop the dependence of norms  on the space $E$ when convenient. We denote by $C^{p-\textnormal{var}}(I; E)$ the set of all paths $z : I \rightarrow E$ such that $\delta z \in C_2^{p-\textnormal{var}}(I; E)$, where  $\delta z_{st} := z_t - z_s$.

For a given interval $I$, a two-index map $\omega: \Delta_I \rightarrow [0,\infty)$ is called superadditive if for all $(s,\theta,t)\in \Delta^{(2)}_I$,
$$
\omega(s,\theta)+\omega(\theta ,t)\le \omega(s,t).
$$
A two-index map $\omega: \Delta_I \rightarrow [0,\infty)$  is called a control if it is superadditive, continuous on $\Delta_I$ and for all $s\in I$, $\omega(s,s)=0$. 

If for a given $p > 0$, $g \in C^{p-\textnormal{var}}_2(I; E)$, then  it can be shown that the two-index map $\omega_g: \Delta_I \rightarrow [0,\infty)$ defined by
$$
\omega_g(s,t)= \|g\|_{p-\textnormal{var};[s,t]}^p 
$$
is a control  (see, e.g.,  Proposition 5.8 in \cite{FrVi10}). 

We shall need a local version of the $p$-variation spaces, for which we restrict the mesh size of the partition  by a control.

\begin{Definition}\label{def:variationSpace}
Given an interval $I=[a,b]$, a control $\varpi$ and real number $L> 0$, we denote by $C^{p-\textnormal{var}}_{2, \varpi, L}(I; E)$  the space of continuous two-index maps $g : \Delta_I \rightarrow E$ for which  there exists at least one control $\omega$ such that for every $(s,t)\in \Delta_I$ with  $\varpi(s,t) \leq L$, it holds that
$
|g_{st}|_E \leq \omega(s,t)^{\frac{1}{p}}.
$ We define a semi-norm on this space by
$$
|g |_{p-\textnormal{var}, \varpi,L; I} =\inf \left\{\omega(a,b)^{\frac{1}{p}} :  \omega  \textnormal{ is a control s.t. } |g_{st}| \leq \omega(s,t)^{\frac{1}{p}}, \;\forall (s, t)  \in \Delta_{I} \textnormal{  with } \varpi(s,t) \leq L \right\} . 
$$ 
\end{Definition}
For a two-index map $g: \Delta_{I}\rightarrow R$, we define the second order increment operator $$\delta g_{s \theta t} = g_{st} - g_{\theta t} - g_{s \theta}, \quad \forall (s,\theta,t)\in\Delta^{(2)}_I.$$

\begin{Definition}\label{defi-rough-path}
Let $K$ be natural and $p\in[2,3)$. A continuous $p$-rough path is a  pair 
\begin{equation}\label{p-var-rp}
\bZ=(Z, \mathbb{Z}) \in C^{p-\textnormal{var}}_2 ([0,T];R^{K}) \times C^{\frac{p}{2}-\textnormal{var}}_2 ([0,T]; R^{K\times K}) 
\end{equation}
that satisfies the Chen's relation 
\begin{equation*}\label{chen-rela}
\delta \mathbb{Z}_{s\theta t}=Z_{s\theta} \otimes Z_{\theta t},  \quad \forall(s, \theta, t) \in \Delta^{(2)}_{[0,T]}.
\end{equation*}
A rough path $\mathbf{Z}=(Z, \mathbb{Z})$ is said to be geometric if it can be obtained as the limit in the   product topology $C^{p-\textnormal{var}}_2 ([0,T];R^{K}) \times C^{\frac{p}{2}-\textnormal{var}}_2 ([0,T]; R^{K\times K})$  of a sequence of rough paths  $\{(Z^{n},\mathbb{Z}^{n})\}_{n=1}^{\infty}$ such that for each $n=1,2,\dots$, 
$$Z^{n}_{st}:= \delta z^{n}_{st} \quad \textnormal{ and } \quad \mathbb{Z}^{n}_{st}:=\int_s^t \delta z^{n}_{s\theta} \otimes \mathrm{d} z^{n}_\theta ,$$
for some smooth paths $z^n:[0,T] \to R^K$, where the  iterated integral is a Riemann integral. 
We denote by $\mathcal{C}^{p-\textnormal{var}}_g([0,T];R^K)$ the set of geometric $p$-rough paths and endow it with the product topology.
\end{Definition}
We will only consider geometric rough paths. Thus, in case of
a Brownian motion, a Stratonovich integral should be used for the construction of the iterated integral if one wishes to lift it to a geometric rough path.
Now we define unbounded rough drivers, which can be regarded as operator valued rough paths with values in a suitable space of unbounded operators.
In what follows, we call a scale any family $(E^{n}, \| \cdot \|_{n})_{ 0\leq n\leq 3}$ of Banach spaces such that $E^{n}$ is continuously embedded into $E^{m}$ for $n\geq m$. For $n\in\{0,1,2,3\}$ we denote by $E^{-n}$ the topological dual of $E^{n}$, and note that, in general, $E^{-0}\neq E^0.$ 
On the scale $(E_n)_{0\leq n\leq 3}$ we require the existence of a family of smoothing operators $(J^\eta)_{\eta\in (0,1)}$ acting on $E_n$ (for $n=1,2$) in such a way that the two following conditions are satisfied:
\begin{equation}\label{eq:reg-J-eta-1}
\lVert  J^\eta-\mathrm{id}\rVert_{\mathcal L(E_m,E_n)} 
\lesssim 
\eta^{m-n} \quad \text{for} \ \ (n,m)\in \{(0,1),(0,2),(1,2)\} \  , 
\end{equation}
\begin{equation}\label{eq:reg-J-eta-2}
\lVert J^\eta\rVert_{\mathcal L(E_n,E_{m})} \lesssim \eta^{-(m-n)} \quad \text{for} \ \ (n,m)\in \{(1,1),(1,2),(2,2),(1,3),(2,3)\} \ .
\end{equation}

\begin{Definition}
\label{def:urd}
Let $p\in [2,3)$ and $T>0$ be given. A continuous unbounded $p$-rough driver with respect to the scale $(E^{n}, \|\cdot \|_{n})_{0\leq n\leq 3}$, is a pair $\mathbf{A} = (A^1,A^2)$ of $2$-index maps such that
 there exists a  control $\omega_A$ on $[0,T]$ such that for every $(s,t)\in \Delta_T$,
\begin{equation}\label{ineq:UBRcontrolestimates}
\| A^1_{st}\|_{\mathcal{L}(E^{-n},E^{-(n+1)})}^p \leq\omega_{A}(s,t) \ \  \text{for}\ \ n\in\{0,1,2\}, \quad
\|A^2_{st}\|_{\mathcal{L}(E^{-n},E^{-(n+2)})}^{\frac{p}{2}} \leq\omega_{A}(s,t) \ \ \text{for}\ \ n\in\{0,1\},
\end{equation}
and   Chen's relation holds true,
\begin{equation}\label{eq:chen-relation}
\delta A^1_{s\theta t}=0,\qquad\delta A^2_{s\theta t}= A^1_{\theta t}A^1_{s\theta},\;\;\forall (s,\theta,t)\in\Delta^{(2)}_T.
\end{equation}
\end{Definition}

Now we consider the rough PDE
\begin{equation}\label{eq:gen}
\dd g_{t} = \mu(\dd t)+ {\bfA}(\dd t) g_{t}   ,
\end{equation}
where $\bfA=(A^1,A^2)$ is an unbounded $p$-rough driver on a scale $(E_n)_{0\leq n\leq 3}$ and the drift $\mu\in C^{1-\rm{var}}(I;E_{-3})$, which possibly also depends on the solution, is continuous and of finite variation. 
A path $g:I \to E_{-0}$ is called a solution (on $I$) of the equation \eqref{eq:gen} provided there exists $q<3$ and  $g^\natural \in C^{\frac{q}{3}-\rm{var}}_{2}(I,E_{-3})$ such that for every $s,t\in I$, $s<t$, and $\varphi\in E_3$,
\begin{equation}\label{eq:gen2}
(\delta g)_{st}(\varphi)=(\delta \mu)_{st}(\varphi)+g_s(\{A^{1,*}_{st}+A^{2,*}_{st}\}\varphi)+g^\natural_{st}(\varphi).
\end{equation}

The following a priori estimate is given in \cite[Cor. 2.11]{DGHT}.

\begin{Proposition}\label{cor:apriori}
Let $p\in[2,3)$ and fix an interval $I\subseteq [0,T]$. Let ${\bfA}=(A^1,A^2)$ be a continuous unbounded $p$-rough driver with respect to a scale $(E_n)_{0\leq n\leq 3}$, endowed with a family of smoothing operators $(J^\eta)_{\eta\in (0,1)}$ satisfying \eqref{eq:reg-J-eta-1} and \eqref{eq:reg-J-eta-2}, and let $\omega_A$ be a  control satisfying \eqref{ineq:UBRcontrolestimates}.
Consider a path $\mu \in C^{1-\textnormal{var}}(I; E_{-3})$ for which there exists a  control $\omega_\mu$ such that for all $s<t\in I$ and $\varphi \in E_3$, 
\begin{equation}\label{cond-mu-cor}
|(\delta \mu)_{st}(\varphi)|\leq \omega_\mu(s,t)\, \|\varphi\|_{E_2} .
\end{equation}
Besides, let $g$ be a solution on $I$  of the equation \eqref{eq:gen} such that $g^\natural \in C^{\frac{p}{3}-\textnormal{var}}_2(I;E_{-3})$.
Then there exists a constant $L=L(p)>0$ such that if $\omega_A(I)\leq L$, one has, for all $s,t\in I$, $s<t$,
\begin{equation}\label{apriori-bound-cor} 
\begin{split}
 \|g^{\natural}_{st}\|_{E_{-3}}
 \lesssim_{q}\|g\|_{L^\infty(I;E_{-0})}\,\omega_A(s,t)^\frac{3}{p}+\omega_\mu(s,t)\omega_A(s,t)^{\frac{3-p}{p}}  .
\end{split}
\end{equation}
\end{Proposition}

\subsection{Navier--Stokes equations driven by rough paths}
\label{sec:roughNS}

We are now aiming to formulate the system under consideration
\eqref{p1}--\eqref{p2} as a rough  equation in the spirit of \eqref{eq:gen} driven by a rough path $\mathbf{Z}$.
Indeed, the system 
can be rewritten as
\[ \partial_t \bfV + {\mu} =\mathbb{Q} \cdot \nabla_x \bfV \dot{\mathbf{Z}}, \]
where $\bfV = (\varrho, \varrho \mathbf{u})$ and $\mu$ contains all the drift part, that is
\begin{align*}
\mu=\begin{pmatrix}
\Div (\vr \vu) \\
	 \Div (\vr \vu \otimes \vu) + \Grad p(\vr) -\Div \mathbb{S}(\Ds \vu) 
\end{pmatrix}
\end{align*}
to be interpreted as an object in $W^{-1,1}(\mt;R^{N+1})$. As in Hofmanov\'a, Leahy, Nilssen \cite[Section~2.5]{HLN} we
rewrite the equation in the rough path form
\begin{equation}
  \delta \bfV_{s t} + \delta \mu_{s t} = A_{s t}^1 \bfV_s + A_{s t}^2 \bfV_s + \bfV_{s
  t}^{\natural}, \label{eq:NSrough}
\end{equation}
where
\[ A^1_{s t} \varphi =\mathbb{Q} \cdot \nabla_x \varphi {Z}_{s t},
   \qquad A^2_{s t}\varphi =\mathbb{Q} \cdot \nabla_x (\mathbb{Q} \cdot \nabla_{x}
   \varphi) \mathbb{Z}_{s t} . \]
The remainder $\bfV^{\natural}_{s t}=(v^{\natural}_{s t},V^{\natural}_{s t})$ is defined impicitly through the equation, that is
\begin{equation}
  \bfV_{s
  t}^{\natural}:=\delta \bfV_{s t} + \delta \mu_{s t} -A_{s t}^1 \bfV_s - A_{s t}^2 \bfV_s . \label{eq:1'}
\end{equation}
It is a two-index map, not an increment, i.e., it depends on two parameters
here denoted by $s, t$. It is required to be sufficiently small so that $A_{s
t}^1 \bfV_s + A_{s t}^2 \bfV_s$ in \eqref{eq:NSrough} is a local approximation of the rough
integral $\int \mathbb{Q} \cdot \nabla_x \bfV \dd \mathbf{Z}$. 
The scale of function spaces is given by the dual spaces $E_{-n}=W^{-n,1}(\mt;R^{N+1})$ and $E_{n}$ being the corresponding pre-duals. The smoothing operators $J^\eta$ as required in \eqref{eq:reg-J-eta-1} and \eqref{eq:reg-J-eta-2} are given e.g. by projections in Fourier space. We are now in the position to give a rigorous formulation of a weak solution
$(\bfu,\varrho)$: 
\begin{itemize}
\item\label{R1} The momentum equation holds in the sense
that the remainder $V_{st}^{\natural}$ given by 
\begin{align}\label{eq:momrough}
\begin{aligned}
\int_{\mt}V_{st}^{\natural}\cdot\bfphi\dx=&\int_{\mt}\big((\varrho \bfu)(t)-(\varrho \bfu)(s)\big)\cdot \bfphi\dx -\int_s^t\int_{\mt}\varrho \bfu\otimes \bfu:\nabla_x \bfphi\dxs
\\
&+\int_s^t\int_{\mt}\mathbb S(\nabla_x \bfu):\nabla_x \bfphi \dxt-\int_s^t\int_{\mt}
p(\varrho)\,\Div \bfphi \dxs
\\&-\int_{\mt}(\varrho \bfu)(s)\otimes\mathbb Q:\nabla_x \bfphi\,Z_{st}\dx\\
&+\int_{\mt}(\varrho \bfu)(s)\otimes\mathbb Q:\nabla_x(\mathbb Q\cdot\nabla_x \bfphi)\,\mathbb Z_{st}\dx
\end{aligned}
\end{align} 
for $\bfphi\in C^\infty(\mt;R^N)$ satisfies $V^\natural \in C^{\frac{q}{3}-\rm{var}}_{2}(I,W^{-3,1}(\mt;R^N))$ for some $q<3$. Moreover we have $\varrho\bfu(0)=\bfq_0$.
\item\label{R2}  The continuity equation holds in the sense that
the remainder $v_{st}^{\natural}$ given by 
\begin{align}\label{eq:conrough}
\begin{aligned}
\int_{\mt}v_{st}^{\natural}\cdot\psi\dx=&\int_{\mt}\big(\varrho(t)-\varrho (s)\big)\cdot \psi\dx -\int_s^t\int_{\mt}\varrho \bfu\cdot\nabla_x \psi\dxs
\\&-\int_{\mt}\varrho (s)\mathbb Q\cdot\nabla_x \psi\,Z_{st}\dx\\
&+\int_{\mt}\varrho (s)\mathbb Q\cdot\nabla_x(\mathbb Q\cdot\nabla_x \psi)\,\mathbb Z_{st}\dx
\end{aligned}
\end{align}
for all $\psi\in C^\infty(\mt)$ satisfies $v^\natural \in C^{\frac{q}{3}-\rm{var}}_{2}(I,W^{-3,1}(\mt))$ for some $q<3$. Moreover, we have $\varrho(0)=\varrho_0$. 
\item \label{R3} The energy inequality is satisfied in the sense that
\begin{align} \label{eq:enerough}
\begin{aligned}
- \int_I &\partial_t \psi \,
\mathscr E \dt+\int_I\psi\int_{\mt}\mathbb S(\nabla_x \bfu):\nabla_x \bfu\dxt\leq
\psi(0) \mathscr E(0)
\end{aligned}
\end{align}
holds for any $\psi \in C^\infty_c([0, T))$.
Here, we abbreviated
$$\mathscr E(t)= \int_{\mathbb{T}^{N}}\Big(\frac{1}{2} \varrho(t) | {\bfu}(t) |^2 + P(\varrho(t))\Big)\dx$$
and the pressure potential is given by $P(\varrho)=\frac{a}{\gamma-1}\varrho^\gamma$.
\end{itemize}

The main result of this section reads as follows.

\begin{Theorem}\label{thm:mainrough}
Assume that we have
\begin{align*}
\frac{|\bfq_0|^2}{\varrho_0}&\in L^1(\mt),\ \varrho_0\in L^{\gamma}(\mt).
\end{align*}
Furthermore, suppose that $\mathbb Q=(\bfQ_k )_{k=1}^K$ with $\bfQ_k\in R^{N\times N}$ and that $\bfZ$ is a geometric $p$-rough path with $K$ natural and $p\in[2,3)$.
 There is a solution $$(\bfu,\varrho)\in L^2(I;W^{1,2}(\mt))\times C_w(\overline I;L^\gamma(\mt))$$ to \eqref{eq:momrough}--\eqref{eq:enerough}. 
\end{Theorem}

Before proving this result, we note that the equation of continuity holds also in a renormalized sense.

\begin{Lemma}
\label{cor:renrough}
Under the assumptions of Theorem \ref{thm:mainrough} the continuity equation holds in the renormalized sense, that is, for $\theta\in C^1([0,\infty))$ with $\theta'(z) \in C_c[0, \infty)$, the remainder $v_{st}^{\theta,\natural}$ given by 
\begin{align}\label{eq:renrough}
\begin{aligned}
\int_{\mt}v_{st}^{\theta,\natural}\cdot\psi\dx=&\int_{\mt}\big(\theta(\varrho(t))-\theta(\varrho (s))\big)\cdot \psi\dx -\int_s^t\int_{\mt}\theta(\varrho) \bfu\cdot\nabla_x \psi\dxs
\\&+\int_s^t\int_{\mt}\big(\theta(\varrho)-\theta'(\varrho)\varrho\big)\,\Div\bfu\psi\dxs-\int_{\mt}\theta(\varrho (s))\mathbb Q\cdot\nabla_x \psi\,Z_{st}\dx\\
&+\int_{\mt}\theta(\varrho (s))\mathbb Q\cdot\nabla_x(\mathbb Q\cdot\nabla_x \psi)\,\mathbb Z_{st}\dx
\end{aligned}
\end{align} 
for all $\psi\in C^\infty(\mt)$ satisfies $v^{\theta,\natural} \in C^{\frac{q}{3}-\rm{var}}_{2}(I,W^{-3,1}(\mt))$ for some $q<3$.
\end{Lemma}

\begin{proof}[Proof of Theorem~\ref{thm:mainrough} and Lemma~\ref{cor:renrough}]
Applying Definition \ref{defi-rough-path} there is $\{(Z^{n},\mathbb{Z}^{n})\}_{n=1}^{\infty}$ such that for each $n=1,2,\dots$, 
$$Z^{n}_{st}:= \delta z^{n}_{st} \quad \textnormal{ and } \quad \mathbb{Z}^{n}_{st}:=\int_s^t \delta z^{n}_{s\theta} \otimes \mathrm{d} z^{n}_\theta $$
for some smooth $z^n:[0,T]\rightarrow R^K.$
For a given $n$ we gain a weak solution $(\bfu_n,\varrho_n)$ to \eqref{eq:mom}--\eqref{eq:ene} with this data by Theorem \ref{thm:main}. Exactly as in Section \ref{sec:6} we  deduce the following convergences from the energy inequality (which is independent of $\mathbb Q$ and $\bfZ^n$):
\begin{align}
\bfu_{n}&\rightharpoonup \bfu\quad\text{in}\quad L^2(I;W^{1,2}(\mt)),\label{eq:convurough}\\
\varrho_{n}&\rightharpoonup^{\ast}\varrho\quad\text{in}\quad L^\infty(I;L^\gamma(\mt)),\label{eq:convrhorough}\\
\varrho_{n}\bfu_{n}&\rightharpoonup {\varrho}  {\bfu}\quad\text{in}\quad L^2(I, L^q(\mt)),\label{conv:rhov2rough}\\
{\varrho}_{n}  {\bfu}_{n}\otimes  {\bfu}_{n}&\rightarrow  {\varrho}  {\bfu}\otimes  {\bfu}\quad\text{in}\quad L^1(I;L^1(\mt)),\label{conv:rhovv2rough}
\end{align}
where $q\in (1,\frac{6\gamma}{\gamma+6})$ is arbitrary.
Also we have higher integrability of the density in the sense that for $0<\Theta\leq  \frac{2}{N}\gamma-1$
\begin{equation}\label{eq:gamma+1'rough}
\int_{I\times \mt}p(\varrho_{n})\varrho_{n}^{\Theta}\,\dd x\,\dd t\leq c
\end{equation}
with a constant independent of $n$, cf. Theorem~\ref{thm:main}. This yields the existence of a function $\overline p$ such that (for a subsequence)
\begin{align}\label{eq:limp'rough}
p_n(\varrho_n)\rightharpoonup\overline p\quad\text{in}\quad L^{1}(I\times\mt).
\end{align}
In order to pass to the limit in equations
 \eqref{eq:mom} and \eqref{eq:con} we aim at applying Proposition \ref{cor:apriori} to control the reminder.
 As a consequence of \eqref{eq:convurough}--\eqref{conv:rhovv2rough} we can
control all terms in the drift $\mu$ and it follows that
\[ \| \delta \mu^n_{s t} \|_{W^{- 1, 1}} \leq c \omega_{\mu} (s, t),\quad \|(\varrho_n\bfu_n,\varrho_n)\|_{L^\infty(I;W^{0,1})}\leq\,c. \]
uniformly in $n$ for some control $\omega_\mu$.
Since $\mathbb Q$ is constant and $\bfZ$ a $p$-rough path we  have
\begin{equation}
  \| A^{n, 1}_{s t} \|_{\mathcal{L} (W^{- k, 1}, W^{- k - 1, 1})} \leqslant
  c | t - s |^{\alpha} \qquad \tmop{for} \qquad k \in \{ 0, 1, 2 \},
  \label{eq:3}
\end{equation}
\begin{equation}
  \qquad \| A^{n, 2}_{s t} \|_{\mathcal{L} (W^{- k, 1}, W^{- k - 2, 1})}
  \leqslant c | t - s |^{2 \alpha} \quad \tmop{for} \quad k \in \{ 0,
  1 \} , \label{eq:4}
\end{equation}
where $\alpha=1/p$. Consequently, we can apply
Proposition \ref{cor:apriori} to infer that
\begin{align}\label{eq:reminderestimate}
\bfV^{n,\natural}\in C_2^{\frac{p}{3}-\rm{var}}(I;W^{-3,1}(\mt;R^{N+1}))
\end{align}
uniformly in $n$. Here $\bfV^{n,\natural}=(V^{n,\natural},v^{n,\natural})$ is the reminder associated to the approximate equation given by
\begin{align}\label{eq:momroughn}
\begin{aligned}
\int_{\mt}V_{st}^{n,\natural}\cdot\bfphi\dx=&\int_{\mt}\big((\varrho_n \bfu_n)(t)-(\varrho_n \bfu_n)(s)\big)\cdot \bfphi\dx -\int_s^t\int_{\mt}\varrho_n \bfu_n\otimes \bfu_n:\nabla_x \bfphi\dxs
\\
&+\int_s^t\int_{\mt}\mathbb S(\nabla_x \bfu_n):\nabla_x \bfphi \dxt-\int_s^t\int_{\mt}
p(\varrho_n)\,\Div \bfphi \dxs
\\&-\int_{\mt}(\varrho_n \bfu_n)(s)\otimes\mathbb Q:\nabla_x \bfphi\,Z^{n}_{st}\dx\\
&+\int_{\mt}(\varrho_n \bfu_n)(s)\otimes\mathbb Q:\nabla_x(\mathbb Q\cdot\nabla_x \bfphi)\,\mathbb Z^n_{st}\dx
\end{aligned}
\end{align} 
for $\bfphi\in C^\infty(\mt;R^3)$ and
\begin{align}\label{eq:conroughn}
\begin{aligned}
\int_{\mt}v_{st}^{n,\natural}\cdot\psi\dx=&\int_{\mt}\big(\varrho_n(t)-\varrho_n (s)\big)\cdot \psi\dx -\int_s^t\int_{\mt}\varrho_n \bfu_n\cdot\nabla_x \psi\dxs
\\&-\int_{\mt}\varrho_n (s)\mathbb Q\cdot\nabla_x \psi\,Z^n_{st}\dx\\
&+\int_{\mt}\varrho_n (s)\mathbb Q\cdot\nabla_x(\mathbb Q\cdot\nabla_x \psi)\,\mathbb Z^n_{st}\dx
\end{aligned}
\end{align}
for $\psi\in C^\infty(\mt)$.

Due to the convergences \eqref{eq:convurough}--\eqref{conv:rhovv2rough} and the reminder estimate \eqref{eq:reminderestimate} we can pass to the limit in the equations. In particular, the passage to the limit in the deterministic terms on the right hand sides follows by classical arguments. Noting that the rough terms are evaluated pointwise in time, their convergence  is obtained from the convergence of $\vr_{n}$ and $\vr_{n}\vu_{n}$ in a space of weakly continuous functions and the convergence of the canonical rough path $(Z^{n},\mathbb{Z}^{n})$ to $(Z,\mathbb{Z})$. Since the right hand sides converge,  we obtain the convergence of the associated left hand sides, i.e., of the remainders $V^{n,\natural}$ and $v^{n,\natural}$. From the uniform bounds for these remainders  in $C^{\frac{p}{3}-\rm{var}}_{2}(I;W^{-3,1}(\mathbb{T}^{N};R^{N+1}))$, we deduce that the limit is also a remainder, i.e. it possesses  finite $\frac{p}3$-variation. 

In particular, we obtain the following.
\begin{itemize}
 \item  The momentum equation holds in the sense
that the remainder $V_{st}^{\natural}$ given by 
\begin{align}\label{eq:apulim'rough}
\begin{aligned}
\int_{\mt}V_{st}^{\natural}\cdot\bfphi\dx=&\int_{\mt}\big((\varrho \bfu)(t)-(\varrho \bfu)(s)\big)\cdot \bfphi\dx -\int_s^t\int_{\mt}\varrho \bfu\otimes \bfu:\nabla_x \bfphi\dxs
\\
&+\int_s^t\int_{\mt}\mathbb S(\nabla_x \bfu):\nabla_x \bfphi \dxt-\int_s^t\int_{\mt}
\overline p\,\Div \bfphi \dxs
\\&-\int_{\mt}(\varrho \bfu)(s)\otimes\mathbb Q:\nabla_x \bfphi\,Z_{st}\dx\\
&+\int_{\mt}(\varrho \bfu)(s)\otimes\mathbb Q:\nabla_x(\mathbb Q\cdot\nabla_x \bfphi)\,\mathbb Z_{st}\dx
\end{aligned}
\end{align} 
for $\bfphi\in C^\infty(\mt;R^3)$ satisfies $V^\natural \in C^{\frac{p}{3}-\rm{var}}_{2}(I,W^{-3,1}(\mt;R^3))$.
\item The continuity equation holds in the sense that
the remainder $v_{st}^{\natural}$ given by 
\begin{align}\label{eq:apvarrholimrough}
\begin{aligned}
\int_{\mt}v_{st}^{\natural}\cdot\psi\dx=&\int_{\mt}\big(\varrho(t)-\varrho (s)\big)\cdot \psi\dx -\int_s^t\int_{\mt}\varrho \bfu\cdot\nabla_x \psi\dxs
\\&-\int_{\mt}\varrho (s)\mathbb Q\cdot\nabla_x \psi\,Z_{st}\dx\\
&+\int_{\mt}\varrho (s)\mathbb Q\cdot\nabla_x(\mathbb Q\cdot\nabla_x \psi)\,\mathbb Z_{st}\dx
\end{aligned}
\end{align}
for all $\psi\in C^\infty(\mt)$ satisfies $v^\natural \in C^{\frac{p}{3}-\rm{var}}_{2}(I,W^{-3,1}(\mt))$.
\end{itemize}

We are left with the task of proving strong convergence of the density.
Using again the $L^\infty$-truncation introduced in \eqref{eq:Tk'} it holds
\begin{align}
 T_k(\varrho_{n})&\rightharpoonup {T}^{1,k}\quad\text{in}\quad C_w(\overline I;L^p( \mt))\quad\forall p\in[1,\infty),\label{eq:Tk1'rough}\\
\big(T_k'(\varrho_{n})\varrho_{n}-T_k(\varrho_{n})\big)\Div \bfu_{n}&\rightharpoonup{T}^{2,k}
\quad\text{in}\quad L^2(I\times\mt),\label{eq:Tk2'rough}
\end{align}
for some limit functions ${T}^{1,k}$ and ${T}^{2,k}$ and we have again
\begin{align}\label{eq:flux'rough}
\begin{aligned}
\int_{I\times\mt}&\big( p(\varrho_{n})-(\lambda+2\mu)\Div \bfu_{n}\big)\,T_k(\varrho_{n})\dxt\\&\longrightarrow\int_{I\times\mt} \big( \overline{p}-(\lambda+2\mu)\Div \bfu\big)\,T^{1,k}\dxt.
\end{aligned}
\end{align}
The equality \eqref{eq:flux'rough} is not effected by the noise as we work
under the additional assumption that $\mathbb Q=(\alpha_k \mathbb I_{N\times N})_{k=1}^K$ with $\alpha_k\in R$.
In fact we have now
\begin{align*}
& \int_I\int_ {\mt}\left[ u^i_n \Big( \vr_n \partial_{x_j} \Del^{-1} \partial_{x_i} [ T_k(\vr_n) Q_{j,k}] - \vr_n Q_{j,k} \partial_{x_j} \Del^{-1} \partial_{x_i} [T_k(\vr_n)] \Big) \right] \partial_t z^{n}_k\dxs\\
& =\int_I\int_ {\mt}Q_{j,k}\left[ u^i_n \Big( \vr_n \partial_{x_j} \Del^{-1} \partial_{x_i} [ T_k(\vr_n) ] - \vr_n  \partial_{x_j} \Del^{-1} \partial_{x_i} [T_k(\vr_n)] \Big) \right] \partial_t z^{n}_k\dxs=0
\end{align*}
and the same for the quantities in the limit. In order to compute the right-hand side of \eqref{eq:flux'rough} we have to apply
the product rule as in Section 4.2 \cite{HLN}
in order to compute the equation for $t\mapsto \int_{\mt}\varrho\bfu\cdot\Delta^{-1}\nabla_x J^\eta(T^{1,k})\dx$ with the mollification $J^\eta$.
This leads to the terms
\begin{align*}
& \int_I\int_ {\mt}\left[ u^i_n \Big( \vr_n \partial_{x_j} \Del^{-1} \partial_{x_i} [ J^\eta(T^{1,k}) Q_{j,k}] - \vr_n Q_{j,k} \partial_{x_j} \Del^{-1} \partial_{x_i} [J^\eta(T^{1,k})] \Big) \right] \dx\,\dd \bfZ^k\\
& =\int_I\int_ {\mt}Q_{j,k}\left[ u^i_n \Big( \vr_n \partial_{x_j} \Del^{-1} \partial_{x_i} [ J^\eta(T^{1,k})] - \vr_n  \partial_{x_j} \Del^{-1} \partial_{x_i} [J^\eta(T^{1,k})] \Big) \right] \dx\,\dd \bfZ^k=0,
\end{align*}
which vanish again. Passing with $\eta\rightarrow 0$ and dealing with the deterministic terms as in the known theory yields \eqref{eq:flux'rough}.

Now, we shall prove that $\varrho$ is a renormalized solution as stated in Lemma~\ref{cor:renrough}.
Reformulating \eqref{eq:ren} with $\theta=T_k$ in the rough path sense, we observe that the  corresponding remainders are again bounded uniformly in $n$ as a consequence of Proposition~\ref{cor:apriori} and the uniform energy estimates. Hence we may pass to the limit $n\rightarrow\infty$ and we arrive at
\begin{align}\label{eq:Tk''rough}
\partial_t T^{1,k}+\Div\big( T^{1,k}\bfu\big)+T^{2,k}= \Div\big(T^{1,k} \mathbb Q\big)\dot{\bfZ}
\end{align}
in the sense of distributions in $I\times\mt$. This equation has to interpreted in the sense that
$v_{st}^{T,\natural}$ given by 
\begin{align}\label{eq:conrough}
\begin{aligned}
\int_{\mt}v_{st}^{T,\natural}\cdot\psi\dx=&\int_{\mt}\big(T^{1,k}(t)-T^{1,k}(s)\big)\cdot \psi\dx -\int_s^t\int_{\mt}T^{1,k} \bfu\cdot\nabla_x \psi\dxs
\\&+\int_s^t\int_{\mt}T^{2,k} \, \psi\dxs-\int_{\mt}T^{1,k} (s)\mathbb Q\cdot\nabla_x \psi\,Z_{st}\dx\\
&+\int_{\mt}T^{1,k} (s)\mathbb Q\cdot\nabla_x(\mathbb Q\cdot\nabla_x \psi)\,\mathbb Z_{st}\dx
\end{aligned}
\end{align}
for all $\psi\in C^\infty(\mt)$ satisfies $v^{T,\natural} \in C^{\frac{p}{3}-\rm{var}}_{2}(I,W^{-3,1}(\mt))$.
Now, we need to apply smoothing in space to obtain the renormalized formulation
 \begin{align}\label{eq:Tk''renrough}
\partial_t \theta(T^{1,k})+\Div\big( \theta(T^{1,k})\bfu\big)+\theta'(T^{1,k})T^{2,k}= \Div\big(\theta(T^{1,k}) \mathbb Q\big)\dot{\bfZ}
\end{align}
in the sense of distributions
for all $\theta\in C^1([0,\infty))$ with $\theta'(z)=0$ for all $z\geq M_\theta$. It has to be interpreted in the sense that
the remainder $v_{st}^{T,\theta,\natural}$ given by 
\begin{align}\label{eq:renrough'}
\begin{aligned}
\int_{\mt}v_{st}^{T,\theta,\natural}\cdot\psi\dx=&\int_{\mt}\big(\theta(T^{1,k}(t))-\theta(T^{1,k} (s))\big)\cdot \psi\dx -\int_s^t\int_{\mt}\theta(T^{1,k}) \bfu\cdot\nabla_x \psi\dxs
\\&+\int_s^t\int_{\mt}\big(\theta(T^{1,k})-\theta'(T^{1,k})T^{1,k}\big)\,\Div\bfu\psi\dxs\\&+\int_s^t\int_{\mt}\theta'(T^{1,k}) T^{2,k}\psi\dxs
-\int_{\mt}\theta(T^{1,k} (s))\mathbb Q\cdot\nabla_x \psi\,Z_{st}\dx\\
&+\int_{\mt}\theta(T^{1,k} (s))\mathbb Q\cdot\nabla_x(\mathbb Q\cdot\nabla_x \psi)\,\mathbb Z_{st}\dx
\end{aligned}
\end{align} 
for all $\psi\in C^\infty(\mt)$ satisfies $v^{T,\theta,\natural} \in C^{\frac{p}{3}-\rm{var}}_{2}(I,W^{-3,1}(\mt))$.
More precisely, we shall mollify the rough formulation of the continuity equation, then apply the It\^o formula \cite[Proposition 7.6]{FrHa14} pointwise in $x$ and  pass to the limit to remove the mollification. As usual for this step, we need a uniform estimate  for the associated remainders. This is a consequence of Proposition~\ref{cor:apriori} and the boundedness of $T_{k}$.

We can now prove a counterpart of \eqref{eq.amplosc''} which is again not effected by the noise and conclude $T^{1,k}\rightarrow \varrho$
and $\theta'(T^{1,k})T^{2,k}\rightarrow0$ as $k\rightarrow\infty$. Using this in \eqref{eq:renrough'} proves Lemma \ref{cor:renrough} -- the renormalized equation of continuity.
The proof of the strong convergence of $\varrho_n$ follows again \cite[Section 7.3]{feireisl1} and is not effected by the noise.
\end{proof}

\section{Stratonovich noise}
\label{sec:strat}

In this section we study the compressible Navier--Stokes system subject to transport noise of Stratonovich-type.
\begin{align}
	\D \vr + \Div (\vr \vu) \dt &= \Div (\vr \mathbb{Q} ) \circ \D \bfW
	\label{p1'} \\ 
	\D (\vr \vu) + \Div (\vr \vu \otimes \vu) \dt + \Grad p(\vr) \dt &= \Div \mathbb{S}(\Ds \vu) \dt + \Div (\vr \vu \otimes \mathbb{Q} ) \circ \D \bfW ,
\label{p2'}
\end{align}
where 
$\bfW=(W_k)_{k=1}^K$ is a collection of standard Wiener processes and $\mathbb Q=(\bfQ_{k})_{k=1}^{K}\subset R^{N}$.\footnote{Our theory would also allow to consider a cylindrical Wiener process provided the sequence $\bfQ_k$ converges to zero rapidly enough.} The stochastic integrals in \eqref{p1'} and \eqref{p2'} are understood in the Stratonovich sense. In the next subsection we give a rigorous meaning to that and define a weak martingale solution to  \eqref{p1'}--\eqref{p2'}. Eventually, we prove its existence by a Wong--Zakai type argument.

\subsection{Stratonovich integration}
Let $(\Omega,\mathcal F,(\mathcal F_t),\mathbb P)$ be a filtered probability space and let $\bfW=(W_k)_{k=1}^K$ be a collection of standard $(\mathcal F_t)$-Wiener processes.
We define the Stratonovich integrals in \eqref{p1'} and \eqref{p2'} by means of the It\^{o}-Stratonovich correction.
First of all we can define the stochastic integrals
\begin{align*}
&\int_0^t\Div (\vr \mathbb{Q} ) \,\D \bfW=\sum_{k=1}^K\int_0^t\Div (\vr \bfQ_k ) \,\D W_k,\\
&\int_0^t\Div (\bfq\otimes \mathbb{Q} ) \,\D \bfW=\sum_{k=1}^K\int_0^t\Div (\bfq\otimes \bfQ_k ) \,\D W_k,
\end{align*}
as It\^{o}-integrals on the Hilbert spaces $W^{-1,2}(\mt,R^N)$ and $W^{-1,2}(\mt)$. Indeed, if $\vr$ and $\bfq$
are $(\mathcal F_t)$ stochastic process taking values in
$C_w([0,T];L^{2N/(N+2)}(\mathbb{T}^{N}))$ and $C_w([0,T];L^{2N/(N+2)}(\mathbb{T}^{N};R^N))$ respectively,
the It\^{o}-integrals 
\begin{align*}
&\int_0^t\langle\Div (\vr \mathbb{Q} ),\psi\rangle \,\D \bfW=-\sum_{k=1}^K\int_0^t\int_{\mt}\vr \bfQ_k \cdot\nabla\psi\dx \,\D W_k,\quad\psi\in W^{1,2}(\mt)\\
&\int_0^t\langle\Div (\bfq\otimes \mathbb{Q} ),\bfphi\rangle \,\D \bfW=-\sum_{k=1}^K\int_0^t\int_{\mt}\bfq\otimes \bfQ_k :\nabla\nabla_{x}\bfphi\dx \,\D W_k,\quad\bfphi\in W^{1,2}(\mt,R^N),
\end{align*}
are well-defined. The corresponding Stratonovich integrals are now defined via the It\^{o}-Stratonovich correction, that is
\begin{align*}
&\int_0^t\int_{\mt}\vr \bfQ_k \cdot\nabla_x\psi\dx \,\circ\D W_k=\int_0^t\int_{\mt}\vr \bfQ_k \cdot\nabla_x\psi\dx \,\D W_k+\frac{1}{2}\Big\langle\Big\langle\int_{\mt}\vr \bfQ_k \cdot\nabla_x\psi\dx,W_k\Big\rangle\Big\rangle_t,\\
&\int_0^t\int_{\mt}\bfq\otimes \bfQ_k :\nabla_x\bfphi\dx \,\circ\D W_k=\int_0^t\int_{\mt}\bfq\otimes \bfQ_k :\nabla_x\bfphi\dx \,\D W_k+\frac{1}{2}\Big\langle\Big\langle\int_{\mt}\bfq\otimes \bfQ_k :\nabla_x\bfphi\dx,W_k\Big\rangle\Big\rangle_t,
\end{align*}
Here $\langle\langle\cdot,\cdot\rangle\rangle_t$ denotes the cross variation. We compute now the cross variations by means of \eqref{p1'} and \eqref{p2'}. We have
\begin{align*}
\int_{\mt}\vr \bfQ_k \cdot\nabla_x\psi\dx&=\dots-\sum_{\ell}\int_0^t\int_{\mt}\vr \bfQ_\ell \cdot\nabla_x(\bfQ_k \cdot\nabla_x\psi)\dx \,\D W_\ell,\quad\psi\in W^{2,2}(\mt),\\
\int_{\mt}\bfq\otimes \bfQ_k :\nabla_{x}\bfphi\dx&=\dots-\sum_{\ell}\int_0^t\int_{\mt}\bfq\cdot (\bfQ_\ell \cdot\nabla_{x}(\bfQ_k\cdot\nabla_{x}\bfphi))\dx \,\D W_\ell,\quad\bfphi\in W^{2,2}(\mt,R^N),
\end{align*}
where the deterministic terms of the equations with quadratic variation zero are hidden in $\dots$. Plugging the previous considerations together we set
\begin{align*}
&\int_0^t\Div (\vr \mathbb{Q} ) \,\circ\D \bfW=\sum_{k=1}^K\int_0^t\Div (\vr \bfQ_k ) \,\D W_k+\frac{1}{2}\sum_{k=1}^K\int_0^t\Div(\bfQ_k \otimes\bfQ_k\nabla_{x}\varrho) \ds,\\
&\int_0^t\Div (\bfq\otimes \mathbb{Q} ) \,\circ\D \bfW=\sum_{k=1}^K\int_0^t\Div (\bfq\otimes \bfQ_k ) \,\D W_k+\frac{1}{2}\sum_{k=1}^K\int_0^t\Div(\bfQ_k\otimes\bfQ_k\nabla_{x}\bfq)\ds,
\end{align*}
to be understood in $W^{-2,2}(\mt)$ and $W^{-2,2}(\mt,R^N)$ respectively.

Now we can define the objects of interest properly. Given initial data $\bfq_0\in L^{2\gamma/(\gamma+2)}(\mt;R^N)$ and $\vr_0\in L^\gamma(\mt)$,\footnote{We could also allow random initial data in form of an initial law with suitable moments.}
a weak martingale solution to \eqref{p1'}--\eqref{p2'} is a multiplet
\begin{align*}
((\Omega,\mathcal F,(\mathcal F_t),\mathbb P),\bfu,\varrho,\bfW)
\end{align*}
with a filtered probability space $(\Omega,\mathcal F,(\mathcal F_t),\mathbb P)$, an $(\mathcal F_t)$-Wiener process $\bfW$ and $(\bfu,\varrho)$ are $(\mathcal F_t)-$adapted\footnote{The velocity field is not a stochastic process in the classical sense and we understand its adaptedness in the sense of random distributions as introduced in \cite[Chap. 2.2]{BFHbook}.}
and satisfy the following: 
\begin{itemize}
\item\label{D1} The momentum equation holds in the sense that\begin{align}\label{eq:momdW}
\begin{aligned}
&\int_{\mt} \varrho \bfu\cdot \bfphi\dx -\int_0^t\int_{\mt}\varrho \bfu\otimes \bfu:\nabla_x \bfphi\dxs
\\
&+\int_0^t\int_{\mt}\mathbb S(\nabla_x \bfu):\nabla_x \bfphi \dxs-\int_0^t\int_{\mt}
p(\varrho)\,\Div \bfphi \dxs
\\&=\int_{\mt}\bfq_0 \cdot \bfphi\dx-\int_0^t\int_{\mt}\varrho \bfu\otimes\mathbb Q:\nabla_x \bfphi\,\bfW\dx\\
&+\sum_{k=1}^K\int_0^t\int_{\mt}\vr\bfu\cdot\Div(\bfQ_k\otimes\bfQ_k\nabla_{x}\bfphi)\dx \ds
\end{aligned}
\end{align} 
$\mathbb P$-a.s. for all $\bfphi\in C^\infty(\mt;R^{N})$ and we have $\varrho\bfu(0)=\bfq_0$.
\item\label{D2}  The continuity equation holds in the sense that
\begin{align}\label{eq:condW}
\begin{aligned}
\int_{\mt}\varrho\psi\dx-
\int_0^t\int_{\mt}\varrho \bfu\cdot\nabla_{x}\psi\dxs&=\int_{\mt}\varrho_0 \psi\dxt\\
&-\int_0^t\int_{\mt}\varrho \mathbb Q\cdot\nabla_x \psi\,\dd \bfW\dx\\
&+\frac{1}{2}\sum_{k=1}^K\int_0^t\int_{\mt}\vr\Div(\bfQ_k \otimes\bfQ_k\nabla_{x}\psi)\dx \ds
\end{aligned}
\end{align}
$\mathbb P$-a.s.  for all $\psi\in C^\infty(\mathbb{T}^{N})$ and we have $\varrho(0)=\varrho_0$. 
\item \label{D3} The energy inequality is satisfied in the sense that
\begin{align} \label{eq:enedW}
\begin{aligned}
- \int_I &\partial_t \psi \,
\mathscr E \dt+\int_I\psi\int_{\mt}\mathbb S(\nabla_x \bfu):\nabla_x \bfu\dxt\leq
\psi(0) \mathscr E(0)
\end{aligned}
\end{align}
holds $\mathbb P$-a.s. for any $\psi \in C^\infty_c([0, T))$.
Here, we abbreviated
$$\mathscr E(t)= \int_{\Omega(t)}\Big(\frac{1}{2} \varrho(t) | {\bfu}(t) |^2 + P(\varrho(t))\Big)\dx$$
and the pressure potential is given by $P(\varrho)=\frac{a}{\gamma-1}\varrho^\gamma$.
\end{itemize}

\subsection{Convergence {\`a} la Wong--Zakai}
In the previous section we showed the existence of a weak solution
$(\mathbf{u}, \varrho) $ to the system
\eqref{p1}--\eqref{p2} driven by a general rough path $\bfZ$. This in particular includes the case of $\bfZ$ being a realization of the Stratonovich lift $(\mathbf{W},\mathbb{W})$ of a $K$-dimensional Brownian motion
$\mathbf{W}$  on some probability space $(\Omega,
\mathcal{F}, \mathbb{P})$. Namely,  for $0 \leqslant s
\leqslant t \leqslant T$ let
\[ \mathbf{W}_{s t} = \delta \mathbf{W}_{s t} := \mathbf{W}_t -
   \mathbf{W}_s, 
   \qquad \mathbb{W}_{s t} := \int_s^t \delta
   \mathbf{W}_{s r} \otimes \circ \mathd \mathbf{W}_r = \left( \int_s^t
   (\mathbf{W}_{i, r} - \mathbf{W}_{i, s}) \circ \mathd \mathbf{W}_{j,
   r} \right)_{i, j = 1, \ldots, K}. \]
Here we denoted by $\circ$ the Stratonovich stochastic integration. Then $(\mathbf{W},\mathbb{W})$ is a random rough path, in the sense that a.s. $(\mathbf{W},\mathbb{W})(\omega)$ is  a geometric $p$-rough path with $p\in[2,3)$.

In other words, the second component of the rough path, i.e. $\mathbb{W}$ is obtained by using probability theory. But once this is done, we can fix $\omega$ from the set of full probability where the Stratonovich integral is defined and apply the results of Section~\ref{s:r} pathwise, i.e. to the rough path $(\mathbf{W},\mathbb{W})(\omega)$. As the proof made use of compactness and relied on taking subsequences which  generally depend on $\omega$,  the resulting solution is not a stochastic process. The goal of this section is to overcome this issue  and to construct an honest probabilistically weak solution,  adapted to the joint canonical filtration generated by the solution and the Brownian motion. This is achieved by combining the rough path analysis from Section~\ref{s:r} with stochastic compactness arguments based on Skorokhod--Jakubowski's representation theorem. The combination of the rough path theory with the stochastic compactness was done in \cite{FHLN}.

In the first step, we observe that from the uniform energy and pressure estimates we obtain tightness by taking expectation. This is permitted due to the following control of the approximate drivers, obtained e.g. via mollifications  $\mathbf{W}^{n}$ of $\mathbf{W}$. It holds
\begin{equation}
  \| A^{n, 1}_{s t} \|_{\mathcal{L} (W^{- k, 1}, W^{- k - 1, 1})} \leqslant
  C_{A^n} | t - s |^{\alpha} \qquad \tmop{for} \qquad k \in \{ 0, 1, 2 \},
  \label{eq:3}
\end{equation}
\begin{equation}
  \qquad \| A^{n, 2}_{s t} \|_{\mathcal{L} (W^{- k, 1}, W^{- k - 2, 1})}
  \leqslant C^2_{A^n} | t - s |^{2 \alpha} \quad \tmop{for} \quad k \in \{ 0,
  1 \} , \label{eq:4}
\end{equation}
where  the constant $C_{A^n}$ is random. By 
\cite[Exercise 10.14]{FrHa14}  we know that for all $q \in [1, \infty)$
\[ \sup_{n \in \mathbb{N}} \mathbb{E} [C^q_{A^n}] < \infty, \]
and there exists a random constant $C (\omega)$ such that
\[ \sup_{n \in \mathbb{N}} C_{A^n} (\omega) \leqslant C (\omega) . \]

As in the proof of Proposition 15 in \cite{FHLN}, the rough driver needs to be included in the compactness argument by enlarging the path space from \cite[Section 4.5]{BFHbook} to 
\[
\mathcal{X} =  \mathcal{X}_{\vr} \times \mathcal{X}_{\vr \vu} \times \mathcal{X}_{\bfu} 
\times\mathcal{X}_{\mathbf{W}} \times \mathcal{X}_{\nu},
\]
where
\begin{align*}
\mathcal{X}_{\vr} &= \big(L^{\gamma + \Theta}(0,T;L^{\gamma + \Theta}(\Q)),w\big) \cap C\big([0,\infty);\big( L^\gamma(\Q),w\big)\big)\cap
C([0,T]; W^{-k,2}(\Q)),\\
\mathcal{X}_{\vr \vu} &= C\big([0,T]; \big(L^{\frac{2\gamma}{\gamma + 1}}(\Q;R^N),w\big)\big)\cap
C([0,T]; W^{-k,2}(\Q;R^N)),\\
\mathcal{X}_{\bfu} &= \left(L^{2}(0,T; W^{1,2}(\Q,R^N)),w\right),\\
\mathcal{X}_{\mathbf{W}} &= C^{p-\rm{var}}_{2}([0,T];R^{K})\times C^{p/2-\rm{var}}_{2}([0,T];R^{K\times K}) ,\\
\mathcal{X}_{\nu} &= (L^\infty((0,T) \times \Q; {\rm Prob}( \R^{N^2+N+1})),w^*).
\end{align*}
With a slight abuse of notation, in the definition of $\mathcal{X}_{\mathbf{W}}$
 we employ the separable versions of the spaces $C^{p-\rm{var}}_{2}([0,T];R^{K})$, $ C^{p/2-\rm{var}}_{2}([0,T];R^{K\times K})$, i.e. the spaces obtained as closure of smooth functions in the $p$-variation and $p/2$-variation norm, respectively. 
 
Let $(\vr_{n},\vu_{n})$ be a solution corresponding to the driver $(\mathbf{W}^{n},\mathbb{W}^{n})$ obtained in Section~\ref{s:8}.  As  in \cite[Section~4.5]{BFHbook}
we can show that the
family of joint laws
\[
\left\{\mathcal{L} \left[ \vr_n,
\vr_n \vu_n,\bfu_n,  (\mathbf{W}^{n},\mathbb{W}^{n}), \delta_{[\vr_n, \vu_n,\nabla\vu_n]}\right] ;\,n\in\mathbb N\right\}
\]
is tight on $\mathcal X$. By the Jakubowski-Skorokhod theorem
\cite{jakubow} we obtain new a new probability space $(\tilde\Omega,\tilde{\mathcal F},\tilde{\mathbb P})$ and a sequence of new random variables
$$\left[\tilde\vr_m,
\tilde\vr_n \tilde\vu_n,\tilde\bfu_n,  (\tilde{\mathbf{W}}_{n},\tilde{\mathbb{W}}_{n}), \tilde\nu_n\right] ,\,n\in\mathbb N,$$
 with values in $\mathcal X$ with the same law as the original ones converging $\mathbb P$-a.s. in the topology of $\mathcal X$ to
 $$\left[\tilde\vr,
\tilde\vr \tilde\vu,\tilde\bfu,  (\tilde{\mathbf{W}},\tilde{\mathbb{W}}), \tilde\nu\right] .$$
Hence it is clear that equations \eqref{eq:mom}--\eqref{eq:ene} continue to hold on the new probability space.
 Moreover, the passage to the limit in the deterministic as well as in the rough terms in the equations proceeds as in Section \ref{sec:roughNS} above.
Finally, the identification of the limit driver as the lift of a Brownian motion on the new probability space can be done as in Proposition 15 in \cite{FHLN}. As a consequence of this together with adaptedness with respect to the joint filtration
\begin{align*}
\tilde{\mathcal F}_t&=\sigma\Big(\sigma \big(\bfr_t\tilde{\varrho},\bfr_t\tilde{\bfu},\bfr_t\tilde{\mathbf{W}}\big)\cup\big\{\mathcal N\in\tilde{\mathcal F};\;\tilde{\p}(\mathcal N)=0\big\}\Big),\quad t\geq0,
\end{align*}
we conclude that the system is solved in the sense of Stratonovich stochastic integration as specified in \eqref{eq:momdW}--\eqref{eq:enedW}. In conclusion, we deduce the following result.

\begin{Theorem}\label{thm:mainstrat}
Assume that we have
\begin{align*}
\frac{|\bfq_0|^2}{\varrho_0}&\in L^1(\mt),\ \varrho_0\in L^{\gamma}(\mt).
\end{align*}
Furthermore, suppose that $\mathbb Q=(\bfQ_k)_{k=1}^K$ with $\bfQ_k\in R^{N}$.
 There is a weak martingale solution 
\begin{align*}
((\Omega,\mathcal F,(\mathcal F_t),\mathbb P),\bfu,\varrho,\bfW)
\end{align*}
 to \eqref{p1'}--\eqref{p2'} in the sense of \eqref{eq:momdW}--\eqref{eq:enedW}. 

Furthermore, the solution is obtained as a limit of solutions to \eqref{p1}--\eqref{p2} driven by 
 smooth approximations $\mathbf{W}^{n}$ of the Brownian motion $\mathbf{W}$.

\end{Theorem}

\end{document}